\documentclass[11pt, reqno]{amsart}

\usepackage[top=1.5in,bottom=1.3in,left=1.3in,right=1.3in,marginparwidth=1.5cm]{geometry}

\usepackage{times}
\usepackage{amsmath, amssymb, amsfonts, latexsym, mdwlist, amsthm}
\usepackage{bm}
\usepackage{caption, subcaption}
\usepackage{url}
\usepackage[all,cmtip]{xy}
\usepackage{mathtools}
\usepackage{graphicx} 
\usepackage{tikz-cd}

\usepackage[bookmarks, colorlinks, breaklinks]{hyperref}

\hypersetup{linkcolor=darkblue,citecolor=darkblue,filecolor=black,urlcolor=darkblue, linktocpage = true}

\definecolor{darkblue}{rgb}{0,0,.6}

\usepackage{todonotes}
\usepackage{comment}

\usepackage[mathscr]{eucal}

\def\cA{\ensuremath{\mathscr A}}
\def\cB{\ensuremath{\mathcal B}}
\def\cC{\ensuremath{\mathscr C}}
\def\cD{\ensuremath{\mathscr D}}

\def\cE{\ensuremath{\mathcal E}}
\def\cF{\ensuremath{\mathcal F}}

\def\cM{\ensuremath{\mathcal M}}

\def\cO{\ensuremath{\mathcal O}}

\def\cT{\ensuremath{\mathcal T}}
\def\cU{\ensuremath{\mathcal U}}

\newcommand{\rC}{\mathrm{C}}
\newcommand{\rD}{\mathrm{D}}

\newcommand{\rG}{\mathrm{G}}

\newcommand{\rH}{\mathrm{H}}
\newcommand{\rh}{\mathrm{h}}
\newcommand{\rK}{\mathrm{K}}

\newcommand{\rM}{\mathrm{M}}

\def\bG{\ensuremath{\mathbf{G}}}

\def\bP{\ensuremath{\mathbf{P}}}
\def\bQ{\ensuremath{\mathbf{Q}}}
\def\bZ{\ensuremath{\mathbf{Z}}}
\def\bR{\ensuremath{\mathbf{R}}}

\def\bC{\ensuremath{\mathbf{C}}}

\def\bv{\ensuremath{\mathbf{v}}}

\def\Br{\mathop{\mathrm{Br}}\nolimits}
\def\Cat{\mathop{\mathrm{Cat}}}
\def\ch{\mathop{\mathrm{ch}}\nolimits}

\def\Coh{\mathop{\mathrm{Coh}}\nolimits}

\def\Db{\mathop{\mathrm{D}^{\mathrm{b}}}\nolimits}
\def\deg{\mathop{\mathrm{deg}}\nolimits}

\def\dim{\mathop{\mathrm{dim}}\nolimits}
\def\Dqc{\mathop{\mathrm{D}_{\mathrm{qc}}}\nolimits}
\def\ev{\mathop{\mathrm{ev}}\nolimits}

\def\Ext{\mathop{\mathrm{Ext}}\nolimits}

\def\Hilb{\mathop{\mathrm{Hilb}}\nolimits}
\def\hom{\mathop{\mathrm{hom}}\nolimits}
\def\Hom{\mathop{\mathrm{Hom}}\nolimits}

\def\id{\mathop{\mathrm{id}}\nolimits}

\def\Knum{\rK_{\num}}

\def\Ku{\mathop{\mathcal{K}u}\nolimits}

\def\num{\mathop{\mathrm{num}}\nolimits}

\def\Pic{\mathop{\mathrm{Pic}}\nolimits}

\def\Qcoh{\mathop{\mathrm{Qcoh}}\nolimits}

\def\rk{\mathop{\mathrm{rk}}}
\def\Sch{\mathop{{\mathrm{Sch}}}}
\def\Set{\mathop{\mathbf{Set}}}
\def\Sing{\mathop{\mathrm{Sing}}}

\def\Spec{\mathop{\mathrm{Spec}}}

\def\Stab{\mathop{\mathrm{Stab}}\nolimits}

\def\Top{\mathop{\mathbf{Top}}}

\newcommand{\Dperf}{\mathrm{D}_{\mathrm{perf}}}

\newcommand{\st}{\mathrm{st}}
\newcommand{\Mod}{\mathrm{Mod}}
\newcommand{\cHom}{\mathcal{H}\!{\it om}}
\newcommand{\Fun}{\mathrm{Fun}}
\newcommand{\Ind}{\mathrm{Ind}}
\DeclareMathOperator{\Map}{Map}

\DeclareMathOperator{\colim}{colim}

\newcommand{\bmu}{\bm{\mu}}

\DeclareMathOperator{\coev}{coev}

\newcommand{\gl}{\mathrm{gl}}

\newcommand{\set}[1]{\left\{ \, #1 \, \right\}}

\newcommand{\sth}{\;\vline\;}

\newcommand{\brK}{\overline{\mathrm{K}}}

\theoremstyle{plain}
\newtheorem{Thm}{Theorem}[section]
\newtheorem{Prop}[Thm]{Proposition}

\newtheorem{Lem}[Thm]{Lemma}
\newtheorem{Cor}[Thm]{Corollary}

\newtheorem{Ques}[Thm]{Question}

\newtheorem*{Ques*}{Question}

\theoremstyle{definition}
\newtheorem{Def}[Thm]{Definition}
\newtheorem{Rem}[Thm]{Remark}

\newtheorem{Ex}[Thm]{Example}

\makeatletter
\newtheoremstyle{italicsname}
 {3pt}
 {3pt}
 {\itshape}
 {}
{\bf}
 {.}
 {.5em}
 {\thmname{#1}\thmnumber{\@ifnotempty{#1}{ }#2}%
 \thmnote{ {\the\thm@notefont(#3)}}}
\makeatother
\theoremstyle{italicsname}

\usepackage{xpatch}
\makeatletter   
\xpatchcmd{\@tocline}
{\hfil\hbox to\@pnumwidth{\@tocpagenum{#7}}\par}
{\ifnum#1<0\hfill\else\dotfill\fi\hbox to\@pnumwidth{\@tocpagenum{#7}}\par}
{}{}
\makeatother 

\makeatletter
\renewcommand\part{%
   \if@noskipsec \leavevmode \fi
   \par
   \addvspace{4ex}%
   \@afterindentfalse
   \secdef\@part\@spart}

\def\@part[#1]#2{%
    \ifnum \c@secnumdepth >\m@ne
      \refstepcounter{part}%
      \addcontentsline{toc}{part}{Part \thepart.\hspace{1em}#1}%
    \else
      \addcontentsline{toc}{part}{#1}%
    \fi
    {\parindent \z@ \raggedright
     \interlinepenalty \@M
     \normalfont
     \ifnum \c@secnumdepth >\m@ne
     \centering 
     \Large\bfseries \partname\nobreakspace\thepart     
       \nobreak. 
     \fi
     \Large \bfseries { #2}%
     \par}%
    \nobreak
    \vskip 3ex
    \@afterheading}
\def\@spart#1{%
    {\parindent \z@ \raggedright
     \interlinepenalty \@M
     \normalfont
     \huge \bfseries #1\par}%
     \nobreak
     \vskip 3ex
     \@afterheading}
\makeatother

\renewcommand{\thepart}{\Roman{part}}

\makeatletter 
\def\l@subsection{\@tocline{2}{0pt}{3pc}{6pc}{}} 
\makeatother

\numberwithin{equation}{section}

\usepackage{paralist}
\setdefaultenum{(1)}{(a)}{}{}
\usepackage{enumitem} 

\makeatletter
\newtheorem*{rep@theorem}{\rep@title}
\newcommand{\newreptheorem}[2]{%
\newenvironment{rep#1}[1]{%
 \def\rep@title{#2 \ref{##1}}%
 \begin{rep@theorem}}%
 {\end{rep@theorem}}}
\makeatother

\def\Z2{\bZ/2}
\def\m2{\bmu_2}

\def\ev{\mathrm{ev}}

\def\op{\mathrm{op}}

\def\Kum{\mathrm{Kum}}

\begin{document}

\author{Laura Pertusi} 
\address{Dipartimento di Matematica ``F.\ Enriques'' \\
Via Cesare Saldini 50 \\
Universit\`a degli Studi di Milano \\
20133 Milano, Italy \smallskip 
}
\email{laura.pertusi@unimi.it}

\title{Noncommutative varieties and stability conditions: an overview}

\begin{abstract} 
Using the modern perspective of noncommutative algebraic geometry we survey some recent progress in the theory of stability conditions and moduli spaces with applications in hyperk\"ahler geometry and classical algebraic geometry.
\end{abstract}

\maketitle

\setcounter{tocdepth}{2}
\setcounter{tocdepth}{1}
\tableofcontents

\section{Introduction}

\subsection*{Background} Vector bundles on elliptic curves were classified by Atiyah in \cite{Atiyah}. In this case, indecomposable vector bundles with coprime rank and degree are parametrized by the curve itself, showing a different behavior compared to the case of the projective line where every vector bundle splits in the sum of line bundles as previously shown by Grothendieck. In general, to get a nice moduli space whose points are vector bundles on a smooth projective variety with some fixed numerical invariants, it is necessary to impose additional conditions ensuring that the set of sheaves one wants to parametrize is not unbounded; see \cite[\S 1.7]{HL:Moduli}.     

In 1962 Mumford made groundbreaking progress, introducing the notion of slope stability for vector bundles and proving that if $C$ is a smooth projective curve with genus $\geq 2$, then there is a smooth projective variety parametrizing slope semistable vector bundles with fixed rank and degree on $C$ \cite{Mumford}. Analogous results were obtained in higher dimension through the notion of Gieseker stability by the work of Gieseker, Maruyama, Simpson \cite{Gieseker, MaruyamaI, MaruyamaII, Simpson}. This also had a strong impact in differential geometry, in relation to irreducible unitary representations of the fundamental group of $C$ due to Narasimhan-Seshadri \cite{NaraSesha} and the higher dimensional generalizations by Donaldson and Uhlenbeck-Yau \cite{Donaldson, UhYau}.

Meanwhile, the notions of derived category of an abelian category and triangulated category introduced by Grothendieck and Verdier became a central topic in algebraic geometry. Since 1980, starting with the work of Beilinson on the case of the projective space, the bounded derived category $\Db(X)$ of coherent sheaves on a smooth projective variety $X$ has been deeply studied as a new invariant, leading to interesting applications in birational geometry and homological mirror symmetry.

In the celebrated works \cite{Bridgeland:Stab, Bridgeland:K3} Bridgeland introduced a notion of stability for complexes in derived categories, as a generalization of slope stability for vector bundles on curves. The original motivation for its establishment was to provide a mathematical framework for the concept of stability defined by Douglas in string theory. A key feature of stability conditions is that the variation of stability of an object behaves in a controlled way, as it can change when crossing walls inside the complex manifold parametrizing stability conditions on $\Db(X)$. This leads to powerful applications in the study of the birational geometry of moduli spaces of sheaves, as well as in hyperk\"ahler geometry, Brill--Noether theory and counting invariants. We suggest \cite{BM:effective} for an ``effective'' survey on this topic.

In this note, we consider the more general setting of categories having properties similar to $\Db(X)$ of a smooth proper variety $X$, but they do not necessarily arise from geometry. The main goal is to introduce the reader to the theory of stability conditions on these noncommutative smooth and proper schemes and to explain several applications related to the study of the geometry of the associated moduli spaces of semistable objects.

\subsection*{Problems and results}
Let $\cC$ be a smooth and proper linear category over a base scheme $S$ (see Definitions~\ref{def:SlinAlex}-\ref{def:smpr}). 
The first difficult problem concerns the construction of stability conditions on $\cC$ (see Definitions~\ref{def:stabcond_Bridgeland}-\ref{definition-stability-families}):
\begin{enumerate}
\item Are there stability conditions on $\cC$ over $S$?
\end{enumerate}
In \S~\ref{sec:stabonDX} we explain the recent developments on the construction of stability conditions when $\cC=\Db(X)$ of a smooth projective variety $X$ over a field. In particular, we recall the conjectural strategy introduced by Bayer, Macrì and Toda \cite{BMT:3folds-BG} using iterated tilting of slope stability and the new results on Calabi--Yau threefolds \cite{CLi:quintic,Soheylaetall}. Then we treat the case of product varieties $X \times C$, where $C$ is a smooth projective curve. In this case, there is an alternative construction of stability conditions due to Y.\ Liu \cite{Liu}. This is crucial for the construction of stability conditions on some hyperk\"ahler manifolds \cite{LMPSZ} and the recent progress of C.\ Li producing stability conditions on any smooth projective variety \cite{Chunyi:realreduction, Chunyi:remark}.

Let us return to the general setting of a smooth and proper noncommutative scheme $\cC$, and assume there is a stability condition $\sigma$ on $\cC$ with respect to some finite rank free abelian group $\Lambda$. For $v \in \Lambda$ denote by $\rM_\sigma(v)$ the associated moduli space of $\sigma$-stable objects in $\cC$ (see Theorem~\ref{thm:moduli}).
\begin{enumerate}
\item[(2)] When is $\rM_\sigma(v)$ non-empty?
\item[(3)] If $\rM_\sigma(v)$ is non-empty, when is it projective? irreducible? smooth? What is the Kodaira dimension of $\rM_\sigma(v)$? 
\end{enumerate}
In \S~\ref{sec:HKgeo} we discuss these questions together with the existence of stability conditions when $\cC$ is a $2$-Calabi–Yau category. The first studied example comes from cubic fourfolds. In this case, there is an associated noncommutative K3 surface $\Ku(Y)$ arising as a semiorthogonal component of $\Db(Y)$ for a cubic fourfold $Y$. The theory of stability conditions with moduli spaces of stable objects therein leads to the construction of new locally complete families of polarized HK manifolds of K3-type by \cite{BLMNPS:families} (see Corollary~\ref{cor:HKfrommoduli_cubic4}). In the paper in preparation \cite{BPPZ} we describe every general polarized HK manifold of Kummer-type as the Albanese fiber of a moduli space of stable objects in a noncommutative abelian surface (see Corollary~\ref{cor:kummerHK} for the precise formulation).

Another interesting source of examples of noncommutative varieties with stability conditions arises from Kuznetsov components of Fano threefolds of Picard rank one and index $1$ and $2$. We discuss these cases in \S~\ref{sec:nccurves} following \cite{LPZ:higher}, where we prove a non-emptiness result for moduli spaces of semistable objects. Further, for cubic threefolds we show that moduli spaces of semistable objects are irreducible and we study the birational geometry of the general fiber of their Abel--Jacobi maps (see Theorem~\ref{thm:cubic3folds})

As the reader may note, all the known examples of noncommutative smooth and proper varieties come from semiorthogonal components of bounded derived categories on smooth projective varieties (see Example~\ref{ex:ncas} for a potential counterexample). Let $X$ be a smooth projective variety and assume there is a semiorthogonal component  $\cC \subset \Db(X)$ with a stability condition $\sigma$. We formulate the following problem to connect the study of $\cC$ with the geometry of $X$ through moduli spaces of stable objects.
\begin{enumerate}
\item[(4)] For some $v \in \Lambda$, relate $\rM_\sigma(v)$ with moduli spaces of stable vector bundles on $X$ or Hilbert schemes of curves of fixed degree and genus.
\end{enumerate}
The above proposal is motivated by positive results for cubic fourfolds and threefolds. In fact, if $Y$ is a cubic fourfold, then the Fano variety of lines in $Y$ and the HK eightfold constructed by Lehn, Lehn, Sorger and van Straten out of twisted cubic curves in $Y$ (if $Y$ does not contain a plane) are moduli spaces of stable objects in $\Ku(Y)$ by \cite{LPZ:twistedcubics}. Furthermore, in \cite{LPZ:ellipticquintics} we construct a compactification of the twisted intermediate Jacobian fibration of $Y$ as a birational model of a singular moduli space of semistable objects in $\Ku(Y)$. We review these results in \S~\ref{sec:modularHK} together with speculations on Hilbert schemes of high-degree rational curves on $Y$ and results on the derived category of moduli spaces $\rM_\sigma(v)$ of stable objects in $\Ku(Y)$ by \cite{MoritzSaket}. 

\subsection*{Plan of the paper}
We assume that the reader is familiar with the theory of triangulated categories.

We work in the framework of noncommutative algebraic geometry based on the foundational works of Lurie and Perry \cite{HA, NCHPD}. This in particular allows us to consider families of smooth and proper categories over a base scheme as we explain in \S\ref{sec:stableinfcat}-\ref{sec:ncschemes}. In \S\ref{sec:stabcondncs} we introduce the theory of stability conditions and moduli spaces of semistable objects on noncommutative smooth and proper schemes following the groundbreaking results in \cite{BLMNPS:families} with contributions from \cite{BPPZ}. The core is \S\ref{sec:resandopen} which contains recent results and applications of the theory of stability conditions on noncommutative smooth and proper schemes, as well as open problems. 

This survey completely omits the study of spaces of stability conditions, as well as the developing theory and applications of atomic sheaves in hyperk\"ahler geometry.

\subsection*{Acknowledgments}
A special thank is for the organizers Izzet Coskun, Emanuele Macrì, Alex Perry, Kevin Tucker and Isabel Vogt of the Bootcamp of the 2025 Algebraic Geometry Summer Research Institute at Colorado State University from where this survey originated, and for the author's group participants Moritz Hartlieb, Amal Mattoo, Lily McBeath, Ana Pavlakovic, Weite Pi, Saket Shah, Yu Shen, Nicolas Vilches.

It is a pleasure to thank Arend Bayer, Chunyi Li, Emanuele Macrì, Alex Perry, Paolo Stellari and Xiaolei Zhao for carefully reading a preliminary version of this paper and for many insightful comments.

\section{Hints on stable $\infty$-categories} \label{sec:stableinfcat}

In this section, we quickly introduce the notion of stable infinity category and the examples arising from homological algebra, following Lurie's work \cite{HA}.

\subsection{Simplicial sets}
The \emph{simplex category} is the category $\bf{\Delta}$ whose objects are linearly ordered sets $[n]=\lbrace 0<1< \dots < n-1 <n \rbrace$ with $n$ a nonnegative integer, and morphisms are order preserving functions of the form $f \colon [n] \to [m]$. 

There are some natural morphisms in $\bf{\Delta}$, defined as follows. For every $n > 0$, $0 \leq i \leq n$, consider the function $d_n^i \colon [n-1] \to [n]$ such that
\begin{equation}
d_n^i(j)=
\begin{cases}
j  & \text{ if } j <i \\
j+1  & \text{ if } j \geq i.
\end{cases}
\end{equation}
Note that $d_n^i$ is an injection sending $[n-1]$ into $[n]$ by skipping the element $i$ in $[n]$. 
For $n \geq 0$, $0 \leq i \leq n$, let $s_n^i \colon [n+1] \to [n]$ be the function satisfying
\begin{equation}
s_n^i(j)=
\begin{cases}
j  & \text{ if } j \leq i \\
j-1  & \text{ if } j > i.
\end{cases}
\end{equation}
Then $s_n^i$ is surjective and reaches the element $i$ in $[n]$ twice. 

\begin{Rem} \label{Rem:propbound&face}
\begin{enumerate}
\item It is possible to visualize the above morphisms by identifying each $[n]$ to the standard topological $n$-simplex 
$$|\Delta^n|=\lbrace (t_0, \dots, t_n) \in [0,1]^{n+1} : \sum_{i=0}^n t_i=1 \rbrace.$$
Then one can see that $d^i_n$ embeds $|\Delta^{n-1}|$ in $|\Delta^n|$ as its $i$-th face, while $s^i_n$ collapses $|\Delta^n|$ into its $(i+1)$-face.
\item \label{Rem:propbound&face_generators} The morphisms $d_n^i$ and $s_n^i$ form a set of generators for the morphisms in $\bf{\Delta}$. Indeed, every $f \colon [n] \to [m]$ decomposes uniquely as $f=d \circ s$, where $s \colon [n] \to [p]$ is the surjective composition $s^{j_t}_{n-t} \circ \dots \circ s^{j_1}_{n-1}$ with $0 \leq j_t < \dots < j_1 <n$ (corresponding to the elements $j$ of $[n]$ such that $f(j)=f(j+1)$), and $d \colon [p] \to [m]$ is the injective composition $d^{i_s}_{m} \circ \dots \circ d^{i_1}_{m-s+1}$ with $0 \leq i_1 < \dots < i_s \leq m$ (corresponding to the elements in $[m]$ which are not in the image of $f$), with $m-s=n-t=p$, see \cite[\S 1.1]{JoyalT:introSHT}. 
\item The following identities can be verified directly:
\begin{align} \label{eq:cosimplicial-identitities}
\begin{split}
d^j_{n} \circ d^i_{n-1} &= d^i_{n} \circ d^{j-1}_{n-1} \quad \text{if } j<i   \\
s^j_{n-1} \circ s^i_n &= s^i_{n-1} \circ s^{j+1}_n \quad \text{if } j<i   \\
s^j_n \circ d^i_{n+1} &= \begin{cases}
d^i_{n} \circ s^{j-1}_{n-1} \, &\text{if } i<j \\
\id \, &\text{if } i=j \text{ or } i=j+1 \\
d^{i-1}_{n} \circ s^j_{n-1} \, &\text{if } i>j+1.
\end{cases}
\end{split}
\end{align}
\end{enumerate}    
\end{Rem}

\begin{Def}
A \emph{simplicial set} is a functor $S \colon \bf{\Delta}^\op \to \Set$.  
\end{Def}

\begin{Rem} \label{Rem:face&degoperators}
Given a simplicial set $S$, we denote by $S_n$ the set $S([n])$. An element of $S_n$ is called an $n$\emph{-simplex}. 

For every $0 \leq i \leq n$, the \emph{face operator of degree} $i$ is $S(d^i_n) \colon S_{n} \to S_{n-1}$, while $S(s^i_n) \colon S_n \to S_{n+1}$ is the \emph{degeneracy operator of degree} $i$. When $S$ is clear from the context, we will simply denote them by $d_i^n$ and $s_i^n$. The identities \eqref{eq:cosimplicial-identitities} yield dual identities for $d_i^n$ and $s_i^n$.
\end{Rem}

\begin{Ex}
Let $n$ be a nonnegative integer. The functor $\Delta^n \colon \bf{\Delta}^\op \to \Set$ defined by 
$$\Delta^n(-)=\Hom_{\bf{\Delta}}(-, [n])$$
is a simplicial set, known as the \emph{standard} $n$-\emph{simplex}. The face and degeneracy operators are defined by pre-composition with $d^i_n$ and $s^i_n$, respectively:
\begin{align*}
d^m_i= \Delta^n(d^i_m) \colon \Hom_{\bf{\Delta}}([m], [n]) \to \Hom_{\bf{\Delta}}([m-1], [n]) ,& \quad f \colon [m] \to [n]  \mapsto f \circ d^i_m \\
s^m_i= \Delta^n(s^i_m) \colon \Hom_{\bf{\Delta}}([m], [n]) \to \Hom_{\bf{\Delta}}([m+1], [n]) ,& \quad f \colon [m] \to [n]  \mapsto f \circ s^i_m
\end{align*}
\end{Ex}

\begin{Rem} \label{Rem:simplex&Yoneda}
By definition $\Delta^n$ is the simplicial set represented by $[n]$. Further, by Yoneda lemma there is a bijection between natural transformations of the form $\Hom_{\bf{\Delta}}(-, [n]) \to S$, with $S$ a simplicial set, and $S([n])=S_n$. As a consequence, a $n$-simplex $\sigma \in S_n$ determines uniquely a morphism of simplicial sets $\sigma \colon \Delta^n \to S$. 
\end{Rem}

\begin{Def}[\cite{Riehl}, Definitions 5.1, 5.2]
\begin{enumerate}
    \item Let $S$ be a simplicial set. A \emph{simplicial subset} of $S$ is a simplicial set $T$ such that $T_n \subset S_n$ for every integer $n \geq 0$, and for every $f \colon [m] \to [n]$ we have that 
    $$T(f)=S(f)|_{T_n} \colon T_n \to T_{m}.$$
    \item The $i$\emph{-th face} $\partial_i\Delta^n$ is the simplicial subset of $\Delta^n$ generated by $d^i_n \in (\Delta^n)_{n-1}=\Hom_{\bf{\Delta}}([n-1], [n])$, i.e.\ a $k$-simplex $\sigma \in (\Delta^n)_k$ belongs to $(\partial_i\Delta^n)_k$ if $\sigma= d^i_n \circ \alpha$ for some $\alpha \in \Hom_{\bf{\Delta}}([k], [n-1])$.
    \item The \emph{boundary} $\partial \Delta^n$ of the standard $n$-simplex $\Delta^n$ is the simplicial subset of $\Delta^n$ generated by $d^0_n, \dots, d_n^n$.
\end{enumerate}    
\end{Def}

\begin{Rem}
Note that a $k$-simplex $\sigma \in (\Delta^n)_{k}= \Hom_{\bf{\Delta}}([k], [n])$ belongs to  $(\partial \Delta^n)_k$ if and only if it is not surjective by Remark~\ref{Rem:propbound&face}\eqref{Rem:propbound&face_generators}.
\end{Rem}

\begin{Ex} \label{Ex:singsimplset}
The previous definitions generalize some basic notions in algebraic topology. 

Given a topological space $X$, we define the functor $\Sing(X) \colon \bf{\Delta}^\op \to \Set$ such that $\Sing(X)_n$ is the set of singular $n$-simplicies in $X$, i.e.\ continuous maps $\sigma \colon |\Delta^n| \to X$, and to every non-decreasing function $f \colon [k] \to [n]$, it associates $\Sing(X)(f) \colon \Sing(X)_n \to \Sing(X)_k$  given by pre-composition with 
$$ |\Delta^k| \to  |\Delta^n|, \quad (t_0, \dots, t_k) \mapsto (\sum_{f(i)=0}t_i, \dots, \sum_{f(i)=n}t_i).$$
This data defines a simplicial set, called the \emph{singular simplicial set} of $X$, together with the face and degeneracy operators obtained by pre-composition. If $\sigma \colon |\Delta^n| \to X$ is a singular simplex, then $d^n_i(\sigma)$ is obtained by pre-composing $\sigma$ with
$$ |\Delta^{n-1}| \to  |\Delta^n|, \quad (t_0, \dots, t_{n-1}) \mapsto (t_0, \dots, t_{i-1}, 0, t_i, \dots, t_{n-1}),$$
and is identified as the restriction of $\sigma$ to its $i$-th face.
\end{Ex}

\begin{Rem} \label{Rem:geomreal}
Given a simplicial set $S$, it is possible to associate a topological space $|S|$, called its \emph{geometric realization}. We refer for instance to \cite[\S I.2]{GoerssJardine} for more details on this construction. 

We only remark that the geometric realizations of the simplicial sets $\Delta^n$, $\partial_i\Delta^n$ and $\partial \Delta^n$ are the topological spaces one would expect, namely the standard topological simplex $|\Delta^n|$, its $i$-th face $|\partial_i\Delta^n|$ and its boundary $|\partial \Delta^n|$, respectively. Furthermore, $\Sing \colon \Top \to \Set_\Delta$ and $|\,\,| \colon \Set_\Delta \to \Top$ form a pair of adjoint functors between the categories of topological spaces and simplicial sets:
\begin{equation} \label{eq:adjSing-georel}
\Hom_{\Top}(|S|, X) \cong \Hom_{\Set_\Delta}(S, \Sing(X)).    
\end{equation}
\end{Rem}

\subsection{Horns and $\infty$-categories}
The following definition is essential to introduce the notion of $\infty$-categories.

\begin{Def} Let $n$ be a nonnegative integer, and fix $0 \leq k \leq n$.
The $k$-th \emph{simplicial horn} $\Lambda^n_k$ is the simplicial subset of $\Delta^n$ generated by $d^0_n, \dots, d^{k-1}_n, d^{k+1}_n, \dots, d^n_n$.   
\end{Def}

\begin{Rem} \label{Rem:georel_horn}
Similarly to Remark~\ref{Rem:geomreal}, the geometric realization of $\Lambda^n_k$ is the union of the faces of $|\Delta^n|$ except the interior of the $k$-th face.   
\end{Rem}

The next result states that the singular simplicial set of a topological space introduced in Example~\ref{Ex:singsimplset} satisfies a horn extension condition.

\begin{Prop} \label{Prop:singTopisKan}
If $X$ is a topological space, then for every $n>0$, $0 \leq k \leq n$, every morphism of simplicial sets $f \colon \Lambda^n_k \to \Sing(X)$ extends to $\Delta^n \to \Sing(X)$ through the inclusion $\Lambda^n_k \hookrightarrow \Delta^n$:
\begin{equation*}
\xymatrix{
\Lambda^n_k \ar[r]^-f \ar@{^{(}->}[d] & \Sing(X) & \text{ for every } 0\leq k\leq n.\\
\Delta^n \ar@{-->}[ru]& &}
\end{equation*}
\end{Prop}
\begin{proof}
Through the adjuction \eqref{eq:adjSing-georel}, we have that $f$ corresponds to $\sigma \colon |\Delta^n| \to X$ in $\Top$, and it is enough to show that $\sigma$ factors through $|\Lambda^n_k| \hookrightarrow |\Delta^n|$. Using Remark~\ref{Rem:georel_horn} one can show that $|\Lambda^n_k|$ is a deformation retract of $|\Delta^n|$. Then the composition $\sigma \circ r \colon |\Delta^n| \to X$, where $r$ is a retraction, gives the required filling. See \cite[Lemma 5.6]{Riehl}. 
\end{proof}

More generally, a simplicial set is a \emph{Kan complex} if it satisfies the above horn extension condition. Another interesting example of Kan complex is the underlying simplicial set of a simplicial group, see \cite[\S I.3]{GoerssJardine} for more details and the relevance of Kan complexes for the study of the homotopy theory of simplicial sets.

In general, a simplicial set needs not to be a Kan complex. For instance, for $n>0$ the standard $n$-simplex $\Delta^n$ does not satisfies the horn extension condition (exercise). The notion of $\infty$-category requires a weaker version of the horn extension condition in Proposition~\ref{Prop:singTopisKan}.

\begin{Def}
An $\infty$-\emph{category} $\cC$ is a simplicial set $\cC \colon \bf{\Delta}^\op \to \Set$ such that for every $n>0$, $0 < k < n$, every morphism of simplicial sets $f \colon \Lambda^n_k \to \cC$ extends to $\Delta^n \to \cC$ through the inclusion $\Lambda^n_k \hookrightarrow \Delta^n$:
\begin{equation*}
\xymatrix{
\Lambda^n_k \ar[r]^-f \ar@{^{(}->}[d] & \cC &  \text{ for every } 0< k<n.\\
\Delta^n \ar@{-->}[ru]& & }
\end{equation*}  
\end{Def}

Clearly, the singular simplicial set of a topological space and more generally Kan complexes are $\infty$-categories. An interesting feature is that it is possible to associate an $\infty$-category to any category, as explained below.

\begin{Ex}
Let $\cC$ be a small category. Let $N(\cC)$ be the simplicial set such that $N(\cC)_n$ is the set of all composable sequences of morphisms  
$$C_0 \to C_1 \to \dots \to C_n$$
with lenght $n$. By definition, $N(\cC)_0$ is the set of objects in $\cC$, while $N(\cC)_1$ is the set of morphisms in $\cC$. In this case, the degeneracy operator $s^n_i$ adds the identity morphism at $C_i$ in a composable sequence of lenght $n$ as above, getting one of lenght $n+1$, while the face operator $d^n_i$ composes the $i$-th and $i+1$-th arrows if $0<i<n$ and leaves out the first or last arrow for $i=0$ or $n$, respectively.

More formally, regarding $[n]$ as a category, define the nerve functor $N \colon \Cat \to \Set_\Delta$ such that 
$$N(\cC)_n=\Hom_{\Cat}([n], \cC)$$
for a category $\cC$ (see \cite[Examples 3.2, 3.6]{Riehl}).

The nerve of a small category is an $\infty$-category  satisfying the additional property that for every $0<k<n$, every $f \colon \Lambda^n_k \to \cC$ extends \emph{uniquely} to $\Delta^n \to \cC$.
Conversely, every $\infty$-category with this property is isomorphic to the nerve of a small category. We refer to \cite[Proposition 1.1.2.2]{HTT} for the proof of this statement.

In general, the extension property does not hold for horns $\Lambda^n_0$ and $\Lambda^n_n$. For example (see \cite[Remark 1.1.2.3]{HTT}), if $n=2$, let $\Lambda^2_2 \to N(\cC)$ be a morphism of simplicial sets, where $\cC$ is a category. This corresponds to the diagram
\begin{equation*}
\xymatrix{
&C_1 \ar[dr]^f & \\
C_0 \ar[rr]^g & & C_2
}
\end{equation*}
where $C_0, C_1, C_2$ are objects in $\cC$ and $f, g \in N(\cC)_1$ are morphisms in $\cC$. Then satisfying the extension condition would require a morphism $C_0 \to C_1$ communting the diagram, which is not always true (e.g.\ if $C_0=C_2$, $g=\id$ and $f$ is not invertible).
\end{Ex}

\subsection{Homotopy category and stability}

Let $\cC$ be an $\infty$-category. It is possible to use some terminology borrowed from category theory.
\begin{itemize}
\item We say that elements in $\cC_0$ are \emph{objects} of $\cC$. Given two objects $X$, $Y \in \cC_0$, a \emph{morphism} from $X$ to $Y$ is a $1$-simplex $f \in \cC_1$ such that $d^1_1(f)=X$ and $d^1_0(f)=Y$. The higher categorical structure is given by simplicies in $\cC_2$, $\cC_3, \dots, \cC_n, \dots$
\item If $X \in \cC_0$, the identity of $X$ is defined as the $1$-simplex $\id_X=s^0_0(X)$. Indeed, $s^0_0 \colon \cC_0 \to \cC_1$ satisfies $d^1_1\circ s^0_0= d^1_0 \circ s^0_0= \id_{\cC_0}$ as noted in Remark~\ref{Rem:face&degoperators}.
\item Given two morphisms $f \colon X \to Y$ and $g \colon Y \to Z$ in $\cC$, there is a morphism $X \to Z$ which is \emph{a composition} of $f$ and $g$. Indeed, the data of $f$ and $g$ as above defines a morphism of simplicial sets $\gamma=(g, \cdot, f) \colon \Lambda^2_1 \to \cC$ such that $\gamma \circ d^0_2=g$ and $\gamma \circ d^2_2=f$ (see \cite[Tag 050F]{Kerodon}). By definition of $\infty$-category, $\gamma$ extends to $\sigma \colon \Delta^2 \to \cC$. By Yoneda lemma (see Remark~\ref{Rem:simplex&Yoneda}), $\sigma$ defines a $2$-simplex $\sigma \in \cC_2$. Then 
\begin{equation} \label{eq:composition}
h:=d^2_1(\sigma) \in \cC_1    
\end{equation}
satisfies $d^1_1(h)=X$, $d^1_0(h)=Z$ by the identities in Remark~\ref{Rem:face&degoperators}. We can represent $\sigma$ as the diagram 
\begin{equation*}
\xymatrix{
&Y \ar[dr]^g & \\
X \ar[ur]^f \ar@{-->}[rr]^h & & Z
}
\end{equation*}
where $f=d^2_2(\sigma)=f$, $g=d^2_0(\sigma)$, $h=d^2_1(\sigma)$. Since the extension $\sigma$ is not unique, $f$ and $g$ do not uniquely determine $h$ (in fact the composition of $f$ and $g$ is defined only up to homotopy, see Definition~\ref{def:homotopic}).
\item Given two objects $X$, $Y \in \cC_0$, the \emph{mapping space} $\Map_\cC(X,Y)$ is the Cartesian square
\begin{equation*}
\xymatrix{
\Map_\cC(X,Y) \ar[r] \ar[d] & \cC_1 \ar[d]^{(d^1_1, d^1_0)} \\
\bullet= \Delta^0 \ar[r]^{X \times Y} & \cC_0 \times \cC_0
}    
\end{equation*}
\end{itemize}

\begin{Def} \label{def:homotopic}
Let $X, Y \in\cC_0$. Two morphisms $f, g \in \Map_\cC(X,Y)$ are \emph{homotopic} if there exists $\sigma \in \cC_2$ such that $d^2_2(\sigma)=f$, $d^2_1(\sigma)=g$, $d^2_0(\sigma)=\id_Y$, as represented in the diagram
\begin{equation*}
\xymatrix{
&Y \ar[dr]^{\id_Y} & \\
X \ar[ur]^f \ar[rr]^g & & Y
}
\end{equation*}
\end{Def}
In this case, we write $f \sim_{\hom}  g$.

\begin{Lem}[\cite{Kerodon}, Tag 003Z]
Let $X, Y \in\cC_0$. Then $\sim_{\hom}$ is an equivalence relation on $\Map_\cC(X,Y)$.    
\end{Lem}

\begin{Def}
Let $\cC$ be an $\infty$-category. The \emph{homotopy category} of $\cC$, denoted by $\rh(\cC)$, is the category whose objects are those of $\cC$, and morphisms between $X$ and $Y \in\cC_0$ are equivalence classes in $\Map_\cC(X,Y) / \sim_{\hom}$.  
\end{Def}

Note that the identity in $\Hom_{\rh(\cC)}(X, X)$ is the equivalence class of $\id_X$. Further, composition is defined uniquely for any triple of objects $X$, $Y$, $Z$ by
$$\Hom_{\rh(\cC)}(X, Y) \times \Hom_{\rh(\cC)}(Y, Z) \to \Hom_{\rh(\cC)}(X, Z), \quad ([f], [g]) \mapsto [g] \circ [f]:=[h]$$
where $h$ is defined in \eqref{eq:composition} (see \cite[Tags 0043, 0048]{Kerodon} for all the verifications).

From the point of view of algebraic geometry and homological algebra, it is important to understand when the homotopy category of an $\infty$-category has a triangulated structure. This leads to the notion of stable $\infty$-category. 

\begin{Def}
Let $\cC$ be an $\infty$-category. Let $f \colon X \to Y$ be a morphism of $\cC$. A \emph{fiber} of $f$ is a Cartesian square  
\begin{equation*}
\xymatrix{
Z \ar[r]^g \ar[d] & X \ar[d]^f \\
0 \ar[r] & Y.
}
\end{equation*}
A \emph{cofiber} of $f$ is a coCartesian square 
\begin{equation*}
\xymatrix{
X \ar[r]^f \ar[d] & Y \ar[d]^g \\
0 \ar[r] & Z.
}  
\end{equation*}
\end{Def}

\begin{Def}
An $\infty$-category $\cC$ is \emph{stable} if it satisfies the following conditions:
\begin{enumerate}
\item $\cC$ has a zero object $0$, i.e.\ an object which is both initial and final.
\item Every morphism in $\cC$ admits a fiber and cofiber. 
\item A square
\begin{equation*}
\xymatrix{
X \ar[r]^f \ar[d] & Y \ar[d]^g \\
0 \ar[r] & Z
}    
\end{equation*}
is Cartesian (\emph{fiber sequence}) if and only if it is coCartesian (\emph{cofiber sequence}).
\end{enumerate}
\end{Def}

The key point is the following beautiful result of Lurie.
\begin{Thm}[\cite{HA}, Theorem 1.1.2.14] \label{thm:hCistriangulated}
Let $\cC$ be a stable $\infty$-category. Then its homotopy category $\rh(\cC)$ has the structure of triangulated category.   
\end{Thm}

We end this section with important examples of stable infinity categories arising from homological algebra.

\begin{Ex}
Let $\cA$ be an abelian category. Denote by $\rC(\cA)$ the category of complexes in $\cA$, which is a dg category. Then one can associate a simplicial set $N_{\text{dg}}(\rC(\cA))$, called the \emph{differential graded nerve} of $\rC(\cA)$, which is a stable $\infty$-category (see \cite[Construction 1.3.1.6]{HA} for the definition of dg nerve and \cite[Propositions 1.3.1.10, 1.3.2.10]{HA} for the stability). Explicitly, objects of $N_{\text{dg}}(\rC(\cA))$, equivalently elements in $N_{\text{dg}}(\rC(\cA))_0$, are complexes in $\cA$, while morphisms of $N_{\text{dg}}(\rC(\cA))$ are morphism of complexes with zero differential. A $2$-simplex is represented by a diagram
\begin{equation*}
\xymatrix{
& Y \ar[dr]^g \ar@{=>}[d]^k & \\
X \ar[ru]^f \ar[rr]_h & & Z
}    
\end{equation*}
where $f$, $g$, $h$ are morphisms of $N_{\text{dg}}(\rC(\cA))$, and $k \in \Hom_{\rC(\cA)}(X,Z)^{-1}$ such that $d(k)=g \circ f-h$, i.e.\ $g \circ f$ and $h$ are homologous.

Assume $\cA$ has enough projective objects. Define the \emph{derived $\infty$-category} $\rD^-(\cA)$ of bounded above complexes in $\cA$ by
\begin{equation*}
\rD^-(\cA):=  N_{\text{dg}}(\rC^-(\cA_{\text{proj}})),  
\end{equation*}
where $\cA_{\text{proj}}$ is the subcategory of projective objects in $\cA$. Analogously, assuming $\cA$ has enough injective objects, then the \emph{derived $\infty$-category} of bounded below complexes is
\begin{equation*}
\rD^+(\cA):=  N_{\text{dg}}(\rC^+(\cA_{\text{inj}})),
\end{equation*}
where $\cA_{\text{inj}}$ is the subcategory of injective objects in $\cA$. Both are stable $\infty$-categories, as $\rC^-(\cA_{\text{proj}})$ and $\rC^+(\cA_{\text{inj}})$ are closed subcategories of $\rC(\cA)$ with respect to shifts and mapping cones (see \cite[Corollary 1.3.2.18]{HA}).

The homotopy categories $\rh(\rD^-(\cA))$, $\rh(\rD^+(\cA))$ are triangulated categories by Theorem~\ref{thm:hCistriangulated}, and in fact coincide by construction with the classical derived categories of bounded above and below complexes, respectively.
\end{Ex}

\begin{Ex}
Assume that $\cA$ is a Grothendieck abelian category \cite[Tag 079A]{stacks-project}. Denote by $W$  the collection of morphisms in $\rC(\cA)$ which are
quasi-isomorphisms of chain complexes. Then the \emph{derived $\infty$-category} of $\cA$ is defined as the localization of $N_{\text{dg}}(\rC(\cA))$ at $W$:
\begin{equation*}
\rD(\cA):= N_{\text{dg}}(\rC(\cA))[W^{-1}].    
\end{equation*}
The localization of an $\infty$-category at a collection of morphisms is well-defined \cite[Remark 1.3.4.2]{HA}. Moreover, $\rD(\cA)$ is stable by \cite[Proposition 1.3.5.9]{HA}.

As a consequence, when $X$ is a scheme, since the abelian category of quasi-coherent $\cO_X$-modules is Grothendieck \cite[Tag 077P]{stacks-project}, it is possible to define
\begin{equation*}
\Dqc(X):=\rD(\Qcoh(X))   
\end{equation*}
as a stable $\infty$-category, enhancing the usual unbounded derived category of $\cO_X$-modules with quasi-coherent cohomology. Let $\Dperf(X) \subset \Dqc(X)$ be the full $\infty$-subcategory of perfect complexes, i.e.\ complexes which are locally quasi-isomorphic to bounded complexes of finite rank locally free sheaves, whose homotopy category is the usual derived category of perfect complexes. By abuse of notation, in the next sections we will denote by $\Dperf(X)$ and $\Dqc(X)$ both the $\infty$-categories and their homotopy categories.
\end{Ex}

\section{Noncommutative schemes} \label{sec:ncschemes}

From now on, every scheme is assumed quasi-compact and separated. We could work in the more general setting of perfect algebraic stacks applying the results in \cite{BenZvi}, but we prefer to keep a down to earth exposition. In this section we introduce the notions of linear categories over a base scheme $S$ which define families of categories over $S$, and of smoothness and properness for $S$-linear categories. We end with a list of examples which will appear in the applications of the next sections. The main reference is \cite{NCHPD}.

\subsection{Linear categories}
Let $\pi \colon X \to S$ be a morphism of schemes. The goal of this section is to generalize the following definition due to Kuznetsov to the setting of stable $\infty$-categories.

\begin{Def}[\cite{Kuz:basechange}] \label{def:SlinKuz}
A triangulated subcategory $\cC$  of $\Dperf(X)$ is \emph{S-linear} if for every $C \in \cC$ and $E \in \Dperf(S)$, then $$C \otimes \pi^*E \in \cC.$$ 
Equivalently, $\cC$ is stable by tensoring with objects pulled back from $\Dperf(S)$.
\end{Def}

In order to define $S$-linear categories in the context of stable $\infty$-categories, we need to introduce some notions. A functor $\cC \to \cD$ between stable $\infty$-categories is \emph{exact} if it carries the zero object to the zero object and fiber sequences to fiber sequences. We denote by $\Fun^{\text{ex}}(\cC, \cD)$ the collection of all exact funtors between $\cC$ and $\cD$. The collection of small idempotent complete stable $\infty$-categories with exact functors is organized in an $\infty$-category denoted $\Cat_\st$ \cite[\S 1.1.4]{HA}. 

By the foundational results in \cite{HA} the notion of symmetric monoidal category and that of commutative algebra object therein are established in the framework of $\infty$-categories. In particular, having a commutative algebra object $A$ of an $\infty$-category $\cC$, one can study the theory of commutative modules over $A$. This gives rise to an $\infty$-category $\Mod_A(\cC)$ of modules over $A$, see \cite[\S 4.4.5]{HA}. 

In the case of $\Cat_\st$ we have the following important properties:
\begin{itemize}
\item $\Cat_\st$ has a closed symmetric monoidal structure with tensor product $\otimes$ characterized by the universal property
$$\Fun^{\text{ex}}(\cC \otimes \cD, \cE) \cong \Fun'(\cC \times \cD, \cE),$$
where $\Fun'(\cC \times \cD, \cE)$ denotes the full subcategory of the category of all functors between $\cC \times \cD$ and $\cE$ generated by those which are exact in each variable.
\item The category $\Dperf(X)$ is a commutative algebra object in $\Cat_\st$ with product given by the tensor product of sheaves.
\end{itemize}
\begin{Rem}
Due to the vastness of the subject, we cannot include more details, but the interested reader may consult \cite[\S 4.8.1-4.8.2]{HA}. Note that the tensor product structure on $\Cat_\st$ is inherited by that on a ``larger'' infinity category denoted $\Pr\Cat_\st$, of presentable stable $\infty$-categories. In fact, there is an embedding functor
$\Ind \colon \Cat_{\st} \to \Pr\Cat_\st$ \cite[\S 5.5.7]{HTT}, such that $\Ind(\Dperf(X)) \cong \Dqc(X)$ (when $X$ is quasi-compact and separated).
\end{Rem}

\begin{Def}[\cite{NCHPD}, Definition 2.3] \label{def:SlinAlex}
Fix a quasi-compact separated base scheme $S$. Let 
$$\mathrm{Cat}_S:=\Mod_{\Dperf(S)}(\mathrm{Cat}_\st)$$
be the $\infty$-category of modules over the commutative algebra object $\Dperf(S) \in \Cat_\st$. An \emph{$S$-linear category} is an object in $\Cat_S$. Given two $S$-linear categories $\cC$ and $\cD$, an \emph{$S$-linear functor} between them is a morphism $\cC \to \cD$ in $\Cat_S$.
\end{Def}

If $\cC \in \Cat_S$, then by definition there is a functor
$$\cC \times \Dperf(S) \to \cC, \quad (C, E) \mapsto C \otimes E$$
giving the module structure on $\cC$.

\begin{Ex} \label{ex:Dperf_Slin}
In the geometric situation, let $\pi \colon X \to S$ be a morphism of schemes. Then $\Dperf(X)$ has a natural structure of $S$-linear category with the $\Dperf(S)$-module structure given by 
$$(C, E) \in \Dperf(X) \times \Dperf(S) \mapsto C \otimes \pi^*E \in \Dperf(X).$$
Further, let $\cC \subset \Dperf(X)$ be a stable $\infty$-subcategory of $\Dperf(X)$. If the $\Dperf(S)$-module structure on $\Dperf(X)$ preserves $\cC$, then $\cC$ inherits the structure of $S$-linear category by 
$$(C, E) \in \cC \times \Dperf(S) \mapsto C \otimes \pi^*E \in \cC.$$
This generalizes Definition~\ref{def:SlinKuz} in the triangulated setting.
\end{Ex}

The $\infty$-category $\Cat_S$ is endowed with a symmetric monoidal structure with tensor product denoted by $\otimes_{\Dperf(S)}$ and unit $\Dperf(S)$ by \cite[Theorem 4.5.2.1]{HA} 

An important consequence is that there is a well-defined notion of base change. Indeed, if $f \colon T \to S$ is a morphism of schemes, then \emph{the base change of $\cC$ along $f$} is the $T$-linear category
\begin{equation*}
\cC_T:=\cC_S \otimes_{\Dperf(S)} \Dperf(T) \in \mathrm{Cat}_T    .
\end{equation*}
This defines a base change functor $\Cat_S \to \Cat_T$. A special situation is when this is applied to a point $s \in S$. The base change along $\Spec(k(s)) \to S$ is denoted
\begin{equation*}
\cC_s:=\cC \otimes_{\Dperf(S)} \Dperf(\Spec(k(s)))   
\end{equation*}
and is called the \emph{fiber of $\cC$ at $s$}.

Intuitively, we can think of $S$-linear categories as families of categories over the base scheme $S$.

\subsection{Smooth and proper linear categories}

In the previous section we have introduced the notion of noncommutative schemes over a base $S$. Let us now introduce the property of being smooth and proper over $S$, explaning its connection with classical smooth and proper schemes.

\begin{Def} \label{def:dualizable}
Let $\cC$ be a symmetric monoidal $\infty$-category with tensor product $\otimes$ and unit $\bf{1}$. An object $C \in \cC$
is \emph{dualizable} if there exists an object $\cD \in \cC$ called the dual of $\cC$, and morphisms 
\begin{align*}
\ev & \colon \cD \otimes \cC \to \bf{1}\\    
\coev & \colon \bf{1} \to \cC \otimes \cD
\end{align*}
called evaluation and coevaluation maps, such that the compositions
\begin{align*}
\cC & \xrightarrow{\coev \otimes \id} \cC \otimes \cD \otimes \cC \xrightarrow{\id \otimes \ev} \cC \\    
\cD & \xrightarrow{\id \otimes \coev} \cD \otimes \cC \otimes \cD \xrightarrow{\ev \otimes \id} \cD
\end{align*}
are equivalent to the identity.
\end{Def}

\begin{Def} \label{def:smpr}
An $S$-linear category $\cC$ is \emph{smooth and proper over} $S$ if $\cC$ is a dualizable object in $\Cat_S$.     
\end{Def}

\begin{Rem}
It is possible to define the notions of smooth, resp.\ proper, $S$-linear category as in \cite[Definition 4.5]{NCHPD}. Together, they are equivalent to Definition~\ref{def:smpr} by \cite[Lemma 4.8]{NCHPD}, which further shows that if $\cC$ is smooth and proper over $S$ then its dual is $\cC^\op$. Since in this survey we are interested in smooth and proper categories, we have decided to adopt this packed version of the definition. It has the advantage of being purely in terms of the symmetric monoidal structure on $\Cat_S$.
\end{Rem}

\begin{Ex} \label{ex:Dperf_smpr}
Assume that $\pi \colon X \to S$ is a smooth and proper morphism of schemes, and consider $\cC=\Dperf(X)$. Then $\cC$ is smooth and proper over $S$. Indeed, consider the opposite category $\Dperf(X)^\op$. Taking the dual induces an equivalence $\Dperf(X)^\op \simeq\Dperf(X)$. Using this and   $\Dperf(X) \otimes_{\Dperf(S)} \Dperf(X) \simeq \Dperf(X \times_S X)$ by \cite[Lemma 2.7]{NCHPD}, define the functors
\begin{align*}
\coev= \Delta_* \circ \pi^* &\colon \Dperf(S) \to \Dperf(X \times_S X)\\
\ev=\pi_* \circ \Delta^* &\colon \Dperf(X \times_S X) \to \Dperf(S).
\end{align*}
where $\Delta \colon X \to X \times_S X$ is the diagonal embedding. By definition $\coev$ is represented by $\Delta_*\cO_X$ and $\ev$ is the restriction to the diagonal composed with $\pi_*$.
As $\pi \colon X \to S$ is smooth and proper, these functors are indeed well-defined on perfect complexes since $\pi_*$ preserves perfect complexes \cite[Example 2.2(a)]{LipNeeman}, and $\Delta$ is a regular closed immersion so $\Delta_*\cO_X$ is perfect. The fact that the first composition
$$\Dperf(X) \xrightarrow{\Delta_* \circ \pi^* \otimes \id} \Dperf(X \times_S X \times_S X) \xrightarrow{\id \otimes \pi_* \circ \Delta^*} \Dperf(X)$$
is equivalent to the identity follows by base change through the diagram
\begin{equation*}
\xymatrix{
X \ar[r]^\Delta \ar[d]^\Delta& X \times_S X \ar[r]^-{\pi \times \id} \ar[d]^{\Delta \times \id} & X \\
X \times_S X \ar[d]^{\id \times \pi} \ar[r]^-{\id \times \Delta}& X \times_S X \times_S X & \\
X & &
}    
\end{equation*}
and similarly for the second.
\end{Ex}

We now see that the collection of smooth and proper $S$-linear categories is well-behaved with respect to base change and taking admissible subcategories.
\begin{Prop} \label{prop:smpr_bc}
Let $\cC$ be a smooth and proper $S$-linear category.
\begin{enumerate}
    \item \cite[Lemma 4.10]{NCHPD} \label{prop:smpr_bc1} If $T \to S$ is a morphism of schemes, then the base change $\cC_T$ is a smooth and proper $T$-linear category.
    \item \cite[Lemma 4.15]{NCHPD} \label{prop:smpr_bc2} Let $\cA$ be an $S$-linear subcategory of $\cC$ such that the inclusion functor $i \colon \cA \to \cC$ is $S$-linear and fully faithful. Assume that $i$ has a left adjoint $i^*$. Then $\cA$ is smooth and proper over $S$.
\end{enumerate}
\end{Prop}
\begin{proof}
Let $\cD$ be the dual of $\cC$ in $\Cat_S$. Then the base change $\cD_T$ and the base change of the evaluation and coevaluation functors for $\cC$ obtained applying $\Cat_S \to \Cat_T$, satisfy the properties of Definition~\ref{def:dualizable}. This proves \eqref{prop:smpr_bc1}.

For \eqref{prop:smpr_bc2} note first that the assumption $i$ is $S$-linear implies that $i$ respects the $\Dperf(S)$-module structures, equivalently the $\Dperf(S)$-module structure on $\cA$ is induced from that on $\cC$. Now $i$ defines $i^\op \colon \cA^\op \to \cC^\op$ between the opposite categories which is fully faithful, $S$-linear and has a left adjoint $(i^{\op})^*$. Note also that the $S$-linear fully faithful functor $i \otimes i^\op \colon \cA \otimes_{\Dperf(S)} \cA^\op \to \cC \otimes_{\Dperf(S)} \cC^\op $ has left adjoint $i^* \otimes (i^\op)^*$ (see \cite[Lemma 2.12]{NCHPD}). Set
\begin{align*}
\alpha &=i^* \otimes (i^{\op})^* \circ \coev \colon \Dperf(S) \to \cA \otimes_{\Dperf(S)} \cA^\op,    \\
\beta&= \ev \circ i^\op \otimes i \colon \cA^\op \otimes_{\Dperf(S)} \cA \to \Dperf(S),
\end{align*}
where $\ev$ and $\coev$ are the evaluation and coevaluation of $\cC$. Then the statement follows from the commutativity of the diagrams
\begin{equation*}
\xymatrix{
\cC \ar[r]^-{\coev \otimes \id} & \cC \otimes_{\Dperf(S)} \cC^\op \otimes_{\Dperf(S)} \cC \ar[d]^-{i^* \otimes (i^\op)^* \otimes i^*} \ar[r]^-{\id \otimes \ev} & \cC \ar[d]^-{i^*}\\
\cA \ar[u]^-i \ar[r]^-{\alpha \otimes \id}& \cA \otimes_{\Dperf(S)} \cA^\op \otimes_{\Dperf(S)} \cA \ar[r]^-{\id \otimes \beta}& \cA,
}    
\end{equation*}
\begin{equation*}
\xymatrix{
\cC^\op \ar[r]^-{\id \otimes \coev} & \cC^\op \otimes_{\Dperf(S)} \cC \otimes_{\Dperf(S)} \cC^\op \ar[d]^-{(i^\op)^* \otimes i^* \otimes (i^\op)^*} \ar[r]^-{\ev \otimes \id} & \cC^\op \ar[d]^-{(i^\op)^*}\\
\cA^\op \ar[u]^-{i^\op} \ar[r]^-{\id \otimes \alpha}& \cA^\op \otimes_{\Dperf(S)} \cA \otimes_{\Dperf(S)} \cA^\op \ar[r]^-{\beta \otimes \id}& \cA^\op,
}    
\end{equation*}
and the identities for $\ev$ and $\coev$.
\end{proof}

\begin{Rem}
Note that $\cA$ as in Proposition~\ref{prop:smpr_bc}\eqref{prop:smpr_bc2} is admissible in $\cC$, i.e.\ $i$ has both left and right adjoint, by \cite[Lemma 4.13]{NCHPD}.     
\end{Rem}

In the next section we will see examples of smooth and proper $S$-linear categories obtained using the previous proposition. Among them, we are interested in an important class of noncommutative schemes, which generalize the notion of Calabi--Yau varieties.

\begin{Def}
\begin{enumerate}
    \item A \emph{noncommutative $n$-Calabi--Yau variety (CY$n$)} over a field $k$ is a $k$-linear category $\cC$ which is smooth and proper over $k$, such that the Serre functor on the homotopy category of $\cC$ is the shift $[n]$.
    \item A \emph{noncommutative $n$-Calabi--Yau variety} over a scheme $S$ is an $S$-linear category $\cC$ which is smooth and proper over $S$, such that for every field valued point $s \in S(k)$ the fiber $\cC_s$ is a noncommutative $n$-Calabi--Yau variety over $k$.
    \end{enumerate}    
\end{Def}

\subsection{Examples} \label{sec:examples}

In Examples~\ref{ex:Dperf_Slin} and \ref{ex:Dperf_smpr} we have seen that if $\pi \colon X \to S$ is a smooth and proper morphism of schemes, then $\Dperf(X)$ is a smooth and proper $S$-linear category. 

Now assume that $\Dperf(X)$ has an \emph{$S$-linear semiorthogonal decomposition} (see for instance \cite{Kuz:surveysod} for an excellent survey in the triangulated setting and \cite[\S 3]{NCHPD} were everything is translated in the $\infty$-categorical language). Roughly speaking, it consists of a collection of $S$-linear categories $\cC_1, \dots, \cC_m$ with the following properties:
\begin{itemize}
\item (Stability under $\otimes$ by $\Dperf(S)$) The natural $\Dperf(S)$-module structure on $\Dperf(X)$ preserves $\cC_i$ for every $i$. Equivalently, the structure of $S$-linear category on $\cC_i$ is that of Example~\ref{ex:Dperf_Slin}.
\item (Semiorthogonality condition) $\pi_*\cHom_X(C_i, C_j)=0$ for every $C_i \in \cC_i$, $C_j \in \cC_j$, $i>j$, where $\cHom_X(-,-)$ denotes the derived sheaf Hom on $X$ (this is the analog condition required in the triangulated setting by \cite[Lemma 2.7]{Kuz:basechange}).
\item (Generation) For every $E \in \Dperf(X)$ there is a diagram
\begin{equation*}
\xymatrix{
0=E_m \ar[rr] & & E_{m-1} \ar[dl] \ar[r]& \dots \ar[r] & E_1 \ar[rr]& & E_0=E  \ar[dl] \\
 & C_m \ar[lu] &  & &  & C_1  \ar[lu] &  
}    
\end{equation*}
with $C_i \in \cC_i$ and the triangles are fiber sequences (which are exact triangles in the homotopy category). 
\end{itemize}
The standard notation is
$$\Dperf(X)=\langle \cC_1, \dots, \cC_m \rangle.$$
In this situation, the components $\cC_i$ are smooth and proper in $\Cat_S$. Indeed, assume first that $m=2$. Since $\cC_1$ is left admissible, by Proposition~\ref{prop:smpr_bc}\eqref{prop:smpr_bc2} it follows that $\cC_1$ is smooth and proper. Since $\cC_1^\perp$ is equivalent to ${}^\perp\cC_1 \simeq \cC_2$, it follows that $\cC_2$ is left admissible and thus smooth and proper by Proposition~\ref{prop:smpr_bc}\eqref{prop:smpr_bc2}. This implies the result for any $m$. See \cite[Lemma 4.15]{NCHPD} for more details.

Let us make some explicit examples of this sort.

\begin{Ex}[Calabi--Yau varieties]
Let $X$ be a smooth projective variety over a field $k$. If $X$ is Calabi--Yau, then the Serre functor of $\Dperf(X)$ satisfies
$$S_{\Dperf(X)}(-)=- \otimes \omega_X[n]=[n]$$
where $n$ is the dimension of $X$. Then $\Dperf(X)$ is a noncommutative CY$n$ variety over $k$. More generally, if $X \to S$ is a smooth and proper family of Calabi--Yau varieties, then $\Dperf(X)$ is a noncommutative CY$n$ variety over $S$.
\end{Ex}

\begin{Ex}[Cubic threefolds]
A cubic threefold $Y$ over a field $k$ is a smooth degree three hypersurface in $\bP^4$. As first noted by Kuznetsov, Kodaira vanishing theorem implies that the line bundles $\cO_Y, \cO_Y(1)$ form an exceptional collection. Thus there is a semiorthogonal decomposition
$$\Dperf(Y)= \langle \Ku(Y), \cO_Y, \cO_Y(1) \rangle$$
where $\Ku(Y):=\langle\cO_Y, \cO_Y(1) \rangle^\perp$ is the right orthogonal complement of the category generated by the exceptional line bundles. Then $\Ku(X)$ is a smooth and proper $k$-linear category. By \cite{kuznetsov-CY} its Serre functor satisfies $S_{\Ku(Y)}^3=[5]$, thus $\Ku(Y)$ is a fractional CY category.

More generally, let $\pi \colon Y \to S$ be a smooth and proper morphism of schemes such that for every point $s \in S$ the fiber $Y_s$ is a cubic threefold. Let $\cO_Y(1)$ be the relative very ample line bundle. Then there is an $S$-linear semiorthogonal decomposition of the form
$$\Dperf(Y)=\langle \Ku(Y), \pi^*\Dperf(S), \pi^*\Dperf(S) \otimes \cO_Y(1) \rangle,$$
where $\Ku(Y)$ is again defined as the residual right orthogonal complement. Indeed, by definition $\pi^*\Dperf(S), \pi^*\Dperf(S) \otimes \cO_Y(1)$ are preserved by the $\Dperf(S)$-module structure. This implies the same holds for $\Ku(Y)$. The semiorthogonality condition follows by base change to any point $s \in S$, and generation follows by construction of the decomposition.
\end{Ex}

\begin{Ex}[Cubic fourfolds] \label{ex:cubicfourfolds}
Similar results hold in the case of (smooth and proper families of) cubic fourfolds. Indeed, there is an $S$-linear semiorthogonal decomposition $$\Dperf(Y)=\langle \Ku(Y), \pi^*\Dperf(S), \pi^*\Dperf(S) \otimes \cO_Y(1), \pi^*\Dperf(S) \otimes \cO_Y(2) \rangle,$$    
where $\cO_Y(1)$ is the relative very ample line bundle of a smooth an proper family $\pi \colon Y \to S$ of cubic fourfolds. If $S=\Spec(k)$ where $k$ is a field, then the Serre functor of $\Ku(Y)$ is the shift $[2]$ by \cite{kuznetsov-CY}. Further, $\Ku(Y)$ has the same Hochschild cohomology as that of the derived category of a K3 surface. In fact, in the moduli space of cubic fourfolds, there are codimension $1$ loci where $\Ku(Y) \simeq \Db(X)$ for a K3 surface $X$, but for $Y$ very general, this does not hold. In this way, $\Ku(Y)$ is a \emph{noncommutative K3 surface}.

In general, since by base change $\Ku(Y)_s \simeq \Ku(Y_s)$ for every point $s \in S$, we have that $\Ku(Y)$ is a noncommutative CY$2$ variety over $S$ and in addition it forms a family of noncommutative K3 surfaces over $S$.
\end{Ex}

\begin{Ex}[Cubic sevenfolds] \label{ex:cubicseven}
Among the other cubic hypersurfaces, it is worth to mention the case of dimension seven. Here $\Ku(Y)= \langle \pi^*\Dperf(S) \otimes \cO_Y(k) \rangle_{k=0, \dots, 5}^\perp$ is a noncommutative CY theerfold over $S$ by \cite[Corollary 4.1]{kuznetsov-CY}.
\end{Ex}

\begin{Ex}[Gushel--Mukai varieties] \label{ex:GMvarieties}
Let $k$ be a field with $\text{char}(k)=0$. A Gushel--Mukai (GM) variety over $k$ of dimension $2 \leq n \leq 6$ is a smooth intersection
$$X=\rC\rG(2,5) \cap Q \subset \bP^{10}$$
where $\rC\rG(2,5)$ is the cone over the Grassmannian $\rG(2,5)$ embedded via the Pl\"ucker embedding in a $10$-dimensional projective space, and $Q$ is a quadric hypersurface in a linear subspace $\bP^{n+4} \subset$ of $\bP^{10}$. By \cite{Kuz-Perry:GM} there is a semiorthogonal decomposition of the form
$$\Dperf(X)= \langle \Ku(X), \cO_X, \cU_X^\vee, \dots, \cO_X(n-3), \cU_X^\vee(n-3) \rangle$$
where $\cU_X$ denotes the restriction of the tautological bundle on the Grassmannian to $X$. According to the dimension, the behavior of $\Ku(X)$ is different: if $n$ is even then $\Ku(X)$ is a noncommutative K3 surface (in fact if $n=2$, then $\Ku(X)=\Dperf(X)$ as $X$ is a K3 surface), while if $n$ is odd then the Serre functor of $\Ku(X)$ satisfies $S_{\Ku(X)}=\iota[2]$, where $\iota$ is an involutive autoequivalence generating a $\bZ/2$-group action on $\Ku(X)$. The latter are examples of \emph{Enriques categories}. Similar results hold in the relative setting over a base. See \cite[Lemma 5.9]{BP:Kuznetsov-conjecture} for details.
\end{Ex}

\begin{Ex}[Twisted sheaves]
Given a smooth and proper morphism of schemes $\pi \colon X \to S$, a natural way of constructing a noncommutative scheme is to fix a Brauer class $\alpha$ on $X$ and consider the derived category of $\alpha$-twisted perfect complexes $\Dperf(X, \alpha)$ on $X$. We refer to \cite{HuybrechtsStellari:Twisted} for definitions and details. Note that $\Dperf(X, \alpha)$ is a smooth and proper $S$-linear category. Indeed, it appears as a semiorthogonal component of a Severi--Brauer scheme by \cite{Bernardara}. This fact has a beautiful application to the period-index problem in the case of abelian threefolds done in \cite{HotchPerry}.
\end{Ex}

\begin{Ex}[Prime Fano threefolds] \label{ex:primeFano3}
Cubic threefolds and GM threefolds are examples of Fano threefolds of Picard rank one. Indeed, for such $X$ let $H$ be a positive generator of $\Pic(X)$. Then the canonical divisor of $X$ satisfies $K_X=-i_XH$, where the positive number $i_X$ is called the index of $X$. In \cite{Isk, MU} a full classification of prime Fano threefolds is provided in terms of the index and the degree $d_X=H^3$. We summarize here the results on $\Dperf(X)$:
\begin{itemize}
\item If $i_X=4$, then $X \cong \bP^4$. Beilinson's result \cite{Beilinson:EquivalencePn} provides a full exceptional collection for $\Dperf(X)$.
\item If $i_X=3$, then $X$ is a quadric threefold. Then $\Dperf(X)$ has a full exceptional collection by Kapranov's result \cite{Kapranov:flagvarieties}.
\item Assume $i_X=2$. In this case, $1 \leq d_X \leq 5$ (for instance when $d_X=3$, then $X$ is a cubic threefold) and by Kodaira vanishing we have the semiorthogonal decomposition
$$\Dperf(X)= \langle \Ku(X), \cO_X, \cO_X(H) \rangle.$$
\item Assume $i_X=1$. Then the degree is even and satisfies $2 \leq d_X  \leq 22$, $d_X \neq 20$. If $d_X \geq 10$ (for instance if $d_X=10$, then $X$ is a GM threefold), then by \cite{BLMS:Kuznetsov, BKM} we have
$$\Dperf(X)= \langle \Ku(X), \cE_X, \cO_X \rangle$$
where $\cE_X$ is a vector bundle on $X$, while if $2 \leq d_X \leq 8$, then
$$\Dperf(X)= \langle \Ku(X), \cO_X \rangle.$$
\end{itemize}
As before, analogous results hold in the relative setting of families of Fano threefolds over a base.
\end{Ex}

So far we have only considered examples included in the following definition.
\begin{Def}
Let $\cC \in \Cat_S$ be smooth and proper over $S$. We say that $\cC$ has \emph{geometric origin} if there is a smooth and proper morphism of schemes $X \to S$ and an $S$-linear fully faithful functor $\cC \to \Dperf(X)$ with left adjoint.
\end{Def}

It would be interesting to understand whether there exist smooth and proper $S$-linear categories which have not geometric origin, as asked in \cite[Question 4.4]{orlov-gluing}. The next example could be a promising case to get a positive answer.

\begin{Ex}[Noncommutative abelian surfaces] \label{ex:ncas}
Let $A$ be a complex abelian surface. Denote by $X$ the Kummer K3 surface obtained as the minimal resolution of the quotient $A /_{-1}$, where $-1$ is the involution defined by the multiplication by $-1$ on $A$. By \cite{BKR}
the derived category of $X$ is identified with the derived category of equivariant sheaves on $A$: there is an equivalence
\begin{equation*}
\Dperf(X) \simeq \Dperf(A)^{\bZ/2}    
\end{equation*}
where $\Dperf(A)^{\bZ/2}$ denotes the invariant category associated to the $\bZ/2$-action on $\Dperf(A)$. See for instance \cite[\S 2]{PPZ_enriques} for more details about group actions on linear categories. In this geometric case, we have $\Dperf(A)^{\bZ/2} \simeq \Dperf([A/(\bZ/2)])$, where $[A/(\bZ/2)]$ is the quotient stack.  

By \cite[Theorem 4.2]{Elagin:equivariant} there is an induced $\bZ/2$-action on $\Dperf(A)^{\bZ/2}$ such that the associated invariant category is equivalent to the original category $\Dperf(A)$. Together with the above equivalence, we get an induced $\bZ/2$-action on $\Dperf(S)$ and an equivalence
\begin{equation*}
\Dperf(X)^{\bZ/2}   \simeq \Dperf(A).  
\end{equation*}
Now the naive idea is to deform $\Dperf(X)$ together with the $\bZ/2$-action and take the invariant category to produce a noncommutative deformation of $\Dperf(A)$.

It turns out that working directly with $X$ does not work. Nevertheless, in \cite{BPPZ} we show that if $T \to S$ is a family of $\Lambda$-polarized K3 surfaces, where $\Lambda$ is a certain lattice of signature $(1, 15)$, with a point $s_0 \in S(\bC)$ such that $\Dperf(T_{s}) \simeq \Dperf(X)$, then the induced $\bZ/2$-action on $\Dperf(T_{s})$ deforms to a $\bZ/2$-action on $\Dperf(T)$. Passing to the invariant category, we obtain a smooth and proper $S$-linear category $\cA$. By \cite[Lemma 6.5]{BO:survey} and base change it follows that $\cA$ is a noncommutative CY$2$ variety. Since the central fiber is $\Dperf(A)$, we can think of $\cA$ as a family of \emph{noncommutative abelian surfaces}.
\end{Ex}

\section{Stability conditions on noncommutative schemes} \label{sec:stabcondncs}

We introduce the notion of stability condition on a smooth and proper $S$-linear category following \cite{BLMNPS:families}. First, we recall the usual notion of stability introduced by Bridgeland when $S$ is a point, then we discuss the relative setting and the results on moduli spaces of semistable objects. 

\subsection{Bridgeland stability conditions}
We start by considering the case of a linear category $\cC$ over a field $k$. The goal of this section is to review the classical notion of stability condition on $\cC$, introduced by Bridgeland in \cite{Bridgeland:Stab}.

\begin{Def}[\cite{HA}, Definition 1.2.1.4]
A \emph{t-structure} $\tau$ on $\cC$ is a t-structure on its homotopy category $\rh(\cC)$. 
\end{Def}

Given a t-structure $\tau$ on $\cC$, we denote by $\cC^{\leq n}$ and $\cC^{\leq n}$ the full subcategories of $\cC$ spanned by the objects in $\rh(\cC)^{\leq n}=\rh(\cC)^{\leq 0}[-n]$ and $\rh(\cC)^{\geq n}=\rh(\cC)^{\geq 0}[-n]$, respectively. The heart of $\tau$ is the full subcategory $\cC^{\heartsuit}=\cC^{\leq 0} \cap \cC^{\geq 0}$. It is an abelian category \cite{Perverse}.
We say that $\tau$ is \emph{bounded} if
$$\cC= \bigcup_{m,n \in \bZ} \cC^{\leq n} \cap \cC^{\geq m}.$$ 

Denote by $\rK_0(\cC)$ the Grothendieck group of $\cC$ and fix a group morphism $v \colon \rK_0(\cC) \to \Lambda$, where $\Lambda$ is a free abelian group of finite rank.

\begin{Def} \label{def:stabcond_Bridgeland}
A \emph{stability condition on $\cC$ with respect to $v$} is a pair $\sigma=(\tau, Z)$, where $\tau$ is a bounded t-structure on $\cC$ and $Z \colon \Lambda \to \bC$ is a group morphism (usually called central charge) satisfying the following properties:
\begin{itemize}
\item $Z \circ v$ is stability function on the heart $\cC^{\heartsuit}$. Explicitly, for every $E \in \cC^{\heartsuit}$, $\Im(Z(E)) \geq 0$ and if equality holds then $\Re(Z(E))<0$;
\item $\cC^{\heartsuit}$ satisfies the Harder--Narasimhan property with respect to the slope defined by $Z$. Explicitly, for every $E \in \cC^{\heartsuit}$, the slope of $E$ with respect to $Z$ is 
$$\mu_Z(E)=
\begin{cases}
-\frac{\Re(Z(E))}{\Im(Z(E))}, \quad \text{if }   \Im(Z(E))>0 \\
+\infty, \quad \text{if }   \Im(Z(E))=0.
\end{cases}
$$
We say that $E$ in $Z$-semistable, resp.\ stable, if for every proper suboject $F \subset E$ we have $\mu_Z(F) \leq \mu_Z(E)$, resp.\ $\mu_Z(F)< \mu_Z(E)$.
Then for every $E \in \cC^{\heartsuit}$ there is a filtration
$$0=E_0 \hookrightarrow E_1 \hookrightarrow \dots \hookrightarrow E_m=E$$
such that $E_i/E_{i-1}$ is $Z$-semistable and $\mu_Z(E_1/E_0) > \dots >\mu_Z(E_m/E_{m-1})$.
\item The support property holds. Explicitly, there is a quadratic form $Q$ on $\Lambda \otimes \bR$ such that $Q|_{\ker(Z)}$ is negative definite, and $Q(E) \geq 0$ for every semistable object $E \in \cC^{\heartsuit}$.
\end{itemize}
\end{Def}

\begin{Ex}[Slope stability on curves]
\label{ex:slopestab}   
Let $X$ be a smooth projective curve. A natural choice for $\Lambda$ is the numerical Grothendieck group of $X$, which is identified with $\bZ \oplus \bZ$ through the rank and the degree. Consider the standard t-structure on $\Db(X)$ whose heart is $\Coh(X)$ and $Z \colon \Lambda \to \bC$ defined by
$$Z(-)= -\deg(-)+ \sqrt{-1}\rk(-).$$
This data defines a stability condition on $\Db(X)$ which generalizes the classical slope stability for vector bundles.
\end{Ex}

\begin{Ex}[Tilt stability on surfaces]
\label{ex:tiltstab} 
Let $X$ be a smooth projective surface. The construction in Example~\ref{ex:slopestab} does not work in this case, since for instance $Z([\cO_x])=0$ for any point $x \in X$. The solution is to change the bounded t-structure through the tilting process. More precisely, fix an ample class $H$ on $X$ and a real number $s$. For every $E \in \Coh(X)$ define the slope
$$\mu_Z(E)=\begin{cases}
-\frac{H^2\ch_1(E)}{H^3\rk(E)} & \text{if } \rk(E)>0 \\
+\infty & \text{otherwise}
\end{cases}$$
and the associated notion of slope stability.
By \cite{Happel-al:tilting} we can consider the t-structure whose heart is 
$\Coh^s(X)= \langle \cT^s, \cF^s[1] \rangle$
given as the extension-closure of the subcategories 
$$\cT^s:=\lbrace E \in \Coh(X): \text{ all slope semistable factors } F \text{ of } E \text{ satisfy } \mu_Z(F)> s \rbrace,$$
$$\cF^s:=\lbrace E \in \Coh(X): \text{ all slope semistable factors } F \text{ of } E \text{ satisfy } \mu_Z(F) \leq s \rbrace.$$
Similarly to Example~\ref{ex:slopestab}, consider the lattice $\Lambda \subset \bQ^3$ spanned by
$$(H^2\rk(E), H\ch_1(E), \ch_2(E)) \quad \text{for every } E \in \Coh(X).$$
For $q \in \bR$ such that $q > \frac{1}{2} s^2$, define $Z_{s, q} \colon \Lambda \to \bC$ by
$$Z_{s,q}(E)=-(\ch_2(E) - qH^2\rk(E))+ \sqrt{-1}(H\ch_1(E)- sH^2\rk(E)).$$
This data defines a continuous family of stability conditions on $\Db(X)$ by \cite{Bridgeland:K3, ArcBertram}. We suggest the beautiful survey \cite{MS:lectures} for details.
\end{Ex}

In the above examples, the existence of well-behaved moduli spaces of semistable objects is known by the classical works of Mumford, Gieseker, Maruyama, Simpson for sheaves and Toda for the case of surfaces (see \cite{Toda:moduliK3} for K3 surfaces which also applies to any surface). However, this is not guaranteed for a stability condition as in Definition~\ref{def:stabcond_Bridgeland}. In the next section, we will see how to overcome this issue, by introducing a new definition of stability condition which works in the relative setting and comes with good moduli spaces. 

\subsection{Stability conditions in families}

The notion of stability condition on a semiorthogonal component of the derived category of a flat and proper family of schemes was introduced in the foundational work \cite{BLMNPS:families}. Roughly speaking, it consists of a collection of stability conditions on each fiber over the base satisfying some gluing and compatibility conditions in order to have well-behaved moduli spaces. We suggest the reader to look at \cite[Definition 1.1]{BLMNPS:families} for an informal but clarifying definition. In \cite{BPPZ} we establish this notion in the setting of smooth and proper linear categories over a base. In particular, we do not require that the category has geometric origin, but we assume smoothness and properness. Here we summarize the key ingredients for the definition.

Let $\cC$ be a smooth and proper linear category over a scheme $S$. The first property a stability condition on $\cC$ over $S$ should enjoy is that its central charge is locally constant in families of objects. If $\cC$ has geometric origin, then by \cite[Definition 21.1, Remark 21.4]{BLMNPS:families} one can define the relative numerical Grothendieck group $\rK_{\num}(\cC/S)$ so that a group morphism $v \colon \rK_{\num}(\cC/S) \to \Lambda$, where $\Lambda$ is a finite rank free abelian group, is universally locally constant: for every scheme $T \to S$ and every $T$-perfect object $E \in \rD(X_T)$ such that $E_t \in \cC_t$ for every $t \in T$ the function $T \to \Lambda$ given by $t \to v([E_t])$ is locally constant. This definition does not work directly in the non-geometric case, since it requires that the numerical Grothendieck group of $\cC$ is finitely generated, which is not known if $\cC$ is smooth and proper. This motivates the following definition.

\begin{Def}[\cite{BPPZ}, Definition 9.5] \label{def:grothgroup}
\begin{itemize}
\item Let $S=\Spec(k)$ where $k$ is a field. 
The \emph{universal Grothendieck group} $\brK_0(\cC)$ of $\cC$ over $k$ is defined by 
\begin{equation*}
\brK_0(\cC) = \colim_{\ell/k} \rK_0(\cC_{\ell}), 
\end{equation*}
where the colimit is taken over the base change for all field extensions $\ell/k$. 
\item The \emph{universally locally constant Grothendieck group} $\brK_0(\cC/S)$ of $\cC$ over $S$ is the quotient of $\bigoplus_{s \in S} \brK_0(\cC_s)$ by the subgroup generated by elements of the form 
$$[E_{t_1}] - [E_{t_2}]$$
for every connected scheme $T \xrightarrow{f} S$, every pair of points $t_1, t_2 \in T$, every object $E \in \cC_T$, where $[E_{t_i}]$ denotes the class of $E_{t_i} \in \cC_{t_i}$ in the group $\brK_0(\cC_{t_i})=\brK_0(\cC_{f(t_i)})$.    
\end{itemize}
\end{Def}

\begin{Rem} \label{Rem:univelcGg_bc}
\begin{enumerate}
    \item \label{Rem:univelcGg_bc1} Note that if $T \to S$ is connected and $E \in \cC_T$, then the class $[E_t] \in \brK_0(\cC/S)$ is constant for every $t \in T$ by definition, so we denote it by $[E]$.
    \item The above definition is compatible with base change: if $f \colon S' \to S$ is a morphism of schemes then the isomorphisms $\brK_0(\cC_{s'}) \cong \brK_0(\cC_{f(s')})$ for every $s' \in S'$ (provided by the colimit construction) induce the morphism 
    \begin{equation*}
    \brK_0(\cC_{S'}/S') \to \brK_0(\cC/S).
    \end{equation*}
\end{enumerate}    
\end{Rem}

Now as before Definition~\ref{def:stabcond_Bridgeland}, fix a group morphism $\bv \colon \brK_0(\cC/S) \to \Lambda$, where $\Lambda$ is a free abelian group of finite rank.

\begin{Def}
A \emph{fiberwise collection of stability conditions on $\cC$ over $S$ with respect to $\bv$} is a pair $\sigma=(\tau, Z)$  where
\begin{itemize}
    \item $\tau=(\tau_s)_{s \in S}$ is a fiberwise collection of t-structures;
    \item $Z \colon \Lambda \to \bC$ is a group morphism such that for every $s \in S$ the pair $\sigma_s=(\tau_s, Z)$ is a stability condition with respect to 
    $$\rK_0(\cC_s) \to \brK_0(\cC_s) \to \brK_0(\cC/S) \xrightarrow{\bv} \Lambda;$$
    \item (Support property) There is a quadratic form $Q$ on $\Lambda \otimes \bR$ such that $Q|_{\ker(Z)}$ is negative definite, and for every $s \in S$ and every $\sigma_s$-semistable object $E \in \cC_s$ we have $Q(E) \geq 0$.
\end{itemize}
\end{Def}

Following \cite{BLMNPS:families} there are conditions which make a fiberwise collection of stability conditions into a family of stability conditions. 

\begin{Def}
\label{definition-stability-families} 
Let $S$ be a Nagata scheme, and $\cC \in \Cat_S$ smooth and proper, satisfying the lifting property over curves. A \emph{stability condition on $\cC$ over $S$ with respect to~$\bv$} is a fiberwise collection of stability conditions $\sigma$ on $\cC$ over $S$ with respect to $\bv$ such that: 
\begin{enumerate} 
\item \label{definition-stability-families_1} $\sigma$ integrates to HN structures over curves,  
\item \label{definition-stability-families_2} 
$\sigma$ satisfies universal openness of geometric stability, 
\item \label{definition-stability-families_3}
$\sigma$ satisfies boundedness of geometrically stable objects.
\end{enumerate}    
\end{Def}

Let us make some comments on the definition:
\begin{itemize}
\item The assumption that $S$ is Nagata simplifies the use of the valuative criterion for universal closedness and properness as explained in \cite[\S 11.4]{BLMNPS:families}. 
\item The lifting property over curves requires that for any essentially of finite type
morphism $C \to S$ where $C$ is a Dedekind scheme, for any open subscheme $U \subset C$ the objects in $\cC_U$ lift to $\cC_C$. This condition holds if $\cC$ has geometric origin, otherwise it should be added as an assumption.
\item Condition~\eqref{definition-stability-families_1} says, roughly speaking, that after base change $C \to S$ as above, the hearts of the base change stability conditions $\sigma_c$ to any point $c \in C$ integrate to a global heart on $C$. This notion generalizes that of relative Harder–Narasimhan filtration and we refer to \cite[\S 13]{BLMNPS:families} for the precise treatment.
\item Condition~\eqref{definition-stability-families_2} holds if for every morphism of schemes $T \to S$ and $E \in \cC_T$, the locus of points $t \in T$ such that $E_t \in \cC_t$  is geometrically $\sigma_t$-stable is open in $T$. Recall that an object in $\cC$ over a field $k$ is geometrically stable if it is stable after base change to $\bar{k}$.
\item We discuss \eqref{definition-stability-families_3} in the next section.
\end{itemize}

\subsection{Moduli spaces}

Let $\cC \in \Cat_S$ be smooth and proper, where $S$ is a Nagata scheme. The moduli stack of gluable objects in $\cC$ is the functor 
\begin{alignat}{3}
\nonumber \cM_{\gl}(\cC/S) \colon  &  (\Sch/S)^{\op} ~~ &  \to & ~~ \mathrm{Grpd} & \\ 
\nonumber & \qquad T & \mapsto 
& \left\{E \in \cC_T \sth  \Ext^{<0}_{k(t)}(E_t, E_t) = 0 \text{ for all } t \in T \right\} & . 
\end{alignat} 
Requiring additionally that $E_T$ is simple, i.e. $\Hom_{k(t)}(E_t, E_t) = k(t)$ for all $t \in T$, defines the subfuctor $s\cM_{\gl}(\cC/S)$ of simple gluable objects.

As shown by Lieblich in the geometric case \cite{Lieblich:mother-of-all} and by the work of To\"{e}n-Vaqui\'{e} in general \cite{toen-moduli}, $\cM_{\gl}(\cC/S)$ is indeed an algebraic stack locally of finite presentation over $S$. Further, $s\cM_{\gl}(\cC/S)$ is an open substack of $\cM_{\gl}(\cC/S)$ which is a $\bG_m$-gerbe over an algebraic space $s\rM_{\gl}(\cC/S)$ locally of finite presentation over $S$.  

Now assume that $\sigma$ is a stability condition on $\cC$ over $S$ with respect to $\bv$ and consider the following substacks of $\sigma$-semistable objects. Let $v \in \Lambda$ and $\phi \in \bR$ with $Z(v) \in \bR_{>0} e^{i\pi\phi}$. Let 
$$\cM_{\sigma}^{\st}(v,\phi) \colon  (\Sch/S)^{\op}  \to  \mathrm{Grpd} $$
be the functor defined by
\begin{equation*}
\cM_{\sigma}^{\st}(v, \phi)(T) 
= \set{  E \in \cC_T \sth 
\begin{array}{c}
\text{$E_t$ is geometrically $\sigma_t$-stable of class} \\ 
\text{$\bv(E_t) = v$ and phase $\phi$ for all $t \in T$}
\end{array} }. 
\end{equation*} 
Similarly define the functor
$$\cM_{\sigma}(v,\phi) \colon  (\Sch/S)^{\op}   \to \mathrm{Grpd} $$
by
\begin{equation*}
\cM_{\sigma}(v,\phi)(T) 
= \set{  E \in \cC_T \sth 
\begin{array}{c}
\text{$E_t$ is $\sigma_t$-semistable of class} \\ 
\text{$\bv(E_t) = v$ and phase $\phi$ for all $t \in T$}
\end{array} }. 
\end{equation*}
Note that $\cM_{\sigma}^{\st}(v,\phi)$ is an open substack of $\cM_{\gl}(\cC/S)$. Indeed, by Remark~\ref{Rem:univelcGg_bc}\eqref{Rem:univelcGg_bc1} for any base change $T \to S$ and $E \in \cC_T$ the function $T \to \Lambda$ given by $t \mapsto \bv([E_t])$ is locally constant, and $\sigma$ satisfies universal openness of geometric stability. In fact, adapting the proof in \cite[Lemma 20.8]{BLMNPS:families} it is possible to show that $\cM_{\sigma}(v,\phi)$ is an open substack of $\cM_{\gl}(\cC/S)$.

Now let us exploit Condition~ \eqref{definition-stability-families_3} of Definition~\ref{definition-stability-families}. Saying that $\sigma$ satisfies boundedness of geometrically stable objects means that the substack $\cM_{\sigma}^{\st}(v, \phi) \subset \cM_{\gl}(\cC/S)$ is bounded for every $v$ and $\phi$, i.e.\ there exist a scheme $B$ of finite type over $S$ and an object $\cE \in \cM_{\sigma}^{\st}(v, \phi)(B)$ such that for every geometric point $\bar{s} \in S$ and $E \in \cM_{\sigma}^{\st}(v, \phi) (k(\bar{s}))$ there exists a $k(\bar{s})$-rational point $b \in B \times_S \Spec(k(\bar{s}))$ such that $\cE_b \cong E$. By \cite[Lemma 21.20]{BLMNPS:families}, which adapts to the setting of smooth and proper linear categories, this condition implies that $\cM_{\sigma}(v,\phi)$ is bounded.

We are ready to state the beautiful result of \cite{BLMNPS:families}.

\begin{Thm}[{\cite[Lemma 21.22 and Theorem 21.24]{BLMNPS:families}}] \label{thm:moduli}
Let $\sigma$ be a stability condition on $\cC$ over $S$ with respect to a group morphism $\bv \colon \brK_0(\cC/S) \to \Lambda$, where $\Lambda$ is a free abelian group of finite rank. Let $v \in \Lambda$ and $\phi \in \bR$ with $Z(v) \in \bR_{>0} e^{i\pi\phi}$.
    \begin{enumerate}
        \item \label{modulithm1} $\cM_{\sigma}^{\st}(v, \phi)$ is an algebraic stack of finite type over $S$, and it is a $\bG_m$-gerbe over an algebraic space $\rM_{\sigma}^{\st}(v, \phi)$ of finite type over $S$.  
        \item \label{modulithm2} $\cM_{\sigma}(v, \phi)$ is an algebraic stack of finite type over $S$, which is universally closed over $S$. If $\cM_{\sigma}(v,\phi) = \cM_{\sigma}^{\st}(v, \phi)$, then $\cM_{\sigma}(v,\phi)$ is a $\bG_m$-gerbe over an algebraic space $\rM_{\sigma}(v,\phi)$ proper over $S$. 
        \item \label{modulithm3} If $S$ has characteristic $0$, then $\cM_{\sigma}(v,\phi)$ admits a good moduli space $\rM_{\sigma}(v, \phi)$ which is an algebraic space proper over $S$. 
    \end{enumerate}
\end{Thm}
\begin{proof}[Idea of proof]
The fact that $\cM_{\sigma}(v, \phi)$ is an algebraic stack of finite type over $S$ follows from the fact that it is an open and bounded substack of $\cM_{\gl}(\cC/S)$ which is an algebraic stack locally of finite presentation over $S$. The same argument applies to $\cM_{\sigma}^{\st}(v, \phi)$ which is made of simple objects, and thus has the claimed $\bG_m$-gerbe structure (see \cite[Corollary 4.3.3]{Lieblich:mother-of-all}).   

Next, Condition~\eqref{definition-stability-families_1} in Definition~\ref{definition-stability-families} allows one to show that $\cM_{\sigma}(v, \phi) \to S$ satisfies the strong existence part of the valuative criterion for any DVR essentially of finite type over S. Since $S$ is Nagata, this implies that $\cM_{\sigma}(v, \phi) \to S$ is universally closed. The existence of a proper good moduli space can be proved using \cite[Theorem 7.25]{AlperHLHein}.
\end{proof}

\begin{Rem}
In the next section, we often simplify the notation for the good moduli space to $\rM_\sigma(v)$ omitting the phase.     
\end{Rem}

\section{Results and open problems} \label{sec:resandopen}

In this section we review some recent developments about the existence of stability conditions on smooth and proper $S$-linear categories and structure results on their moduli spaces of semistable objects, together with applications in hyperk\"ahler geometry and open questions.

\subsection{Stability conditions on $\Dperf(X)$} \label{sec:stabonDX}

One of the first foundational questions in the theory of stability conditions concerns their existence, even in the weaker formulation of Definition~\ref{def:stabcond_Bridgeland}. As seen in Examples~\ref{ex:slopestab} and \ref{ex:tiltstab}, this problem is solved for the bounded derived category of coherent sheaves on smooth projective curves and surfaces, but if $X$ has higher dimension this is more subtle. 

\subsubsection{Tilting method}
The first strategy to attack this problem is due to Bayer, Macrì and Toda \cite{BMT:3folds-BG}. We quickly explain the idea, referring to the beautiful survey \cite{MS:lectures} for details and applications. Assume that $X$ has dimension three, and consider the tilt stability $\sigma_{s,q}$. Tilting the heart $\Coh^s(X)$ with respect to $\sigma_{s,q}$, in a similar way as done for surfaces (see Example~\ref{ex:tiltstab}), produces a new heart $\cA^{s,q}$. It is possible to build a candidate central charge involving also the third Chern character of objects, and checking that it produces a stability function on $\cA^{s,q}$ relies on proving that semistable objects satisfy a quadratic inequality involving their Chern character. This has been successfully achieved in the following cases:
\begin{itemize}
    \item Fano threefolds \cite{BMT:3folds-BG, BMSZ:Fano3folds, Chunyi:stability-Fano-threefolds, Macri:P3, Benjamin:quadric},
    \item abelian threefolds \cite{BMS:stabCY3s, Dulip-Antony:I, Dulip-Antony:II}, and some resolutions of finite quotients of abelian threefolds \cite{BMS:stabCY3s},
    \item threefolds with nef tangent bundle \cite{Koseki_nef, BMSZ:Fano3folds}
    \item quintic threefolds \cite{CLi:quintic},
    \item some three-dimensional weighted hypersurfaces in weighted projective spaces \cite{Koseki_triple}, 
    \item three-dimensional complete intersections of quartic and quadric hypersurfaces in $\bP^5$ \cite{Shengxuan}.
\end{itemize}
A more general conjectural version of the Bayer-Macrì-Toda inequality has been stated in \cite[Conjecture 4.7]{BM:effective}, motivated by the results in \cite[Theorem 1.1]{BMSZ:Fano3folds}. Note also that the stability conditions constructed with this tilting procedure define stability conditions in the sense of Definition~\ref{definition-stability-families}, even in the relative setting of flat families of smooth projective varieties by \cite[Theorem 1.3]{BLMNPS:families}.

If $X$ is a polarized Calabi--Yau threefold, an important progress has been achieved in \cite{Soheylaetall}, where the authors reduce the problem of checking the conjectural inequality to the condition that curves in $X$ satisfy a Brill--Noether-type inequality. More precisely, if $C \subset X$ is a smooth projective curve of arithmetic genus $g$, they define the \emph{Brill–Noether number} of $C$ as
\[
\text{BN}_C=\lim_{t \to 0} \sup \left\{ \frac{\dim\rH^0(E)}{\mathrm{rk}(E)} \;\middle|\;
\begin{array}{l}
E \text{ is a stable sheaf on } C, \\
\frac{\deg(E)}{\mathrm{rk}(E)} \in (g-1-t,\, g-1+t)
\end{array}
\right\}
\]
\begin{Thm}[\cite{Soheylaetall}, Theorem 1.3]
\label{thm:soheyla}
Let $X$ be a Calabi--Yau threefold with ample class $H$. Suppose that there exists a smooth surface $S \in |H|$ and a smooth curve $C \in |H|_S|$ such that $\emph{BN}_C < \chi(\cO_X(H))$. Then the generalized Bayer-Macrì-Toda inequality holds. In particular, there exists a family of geometric stability conditions on $X$.
\end{Thm}
The condition in the above thereom holds for many classes of CY threefolds, such as complete intersections in weighted projective spaces or anticanonical divisors in Fano fourfolds of index $\geq 3$ or index $2$ and Picard rank one, and provides a uniform method to treat the generalized Bayer-Macrì-Toda inequality in this context. It would be interesting to apply Theorem~\ref{thm:soheyla} to other collections of CY threefolds. We mention the following open problem.

\begin{Ques}[\cite{Soheylaetall}, Conjecture 5.9]
Let $X$ be a Calabi--Yau threefold with very ample class $H$ and $\Pic(X)=\bZ H$. Do there exist a  a smooth surface $S \in |H|$ and a smooth curve $C \in |H|_S|$ such that $\emph{BN}_C < \chi(\cO_X(H))$?    
\end{Ques}

\subsubsection{Products with curves}
In the special case of product varieties over a field a new construction is due to Y.\ Liu in \cite{Liu}. This method produces stability conditions on product varieties with curves. More precisely, let $X$ and $S$ be smooth projective varieties over a field $k$.  Assume that $\sigma=(\tau, Z)$ is a stability condition as in Definition~\ref{def:stabcond_Bridgeland} on $\Dperf(X)=\Db(X)$ with respect to $v \colon \rK_0(X) \to \Lambda$, such that $\tau$ has noetherian heart $\cA$. The starting point is the construction of a bounded t-structure on $\Db(X \times S)$ by base change whose noetherian heart is  
$$\cA_S=\lbrace E \in \Db(X \times S): p_*(E \otimes q^*(\cO_S(n))) \in \cA \text{ for every } n \gg 0 \rbrace,$$
where $p$ and $q$ are the projections from $X \times S$ to $X$ and $S$, respectively, and $\cO_S(1)$ is a fixed ample line bundle on $S$. If $S$ is a curve, then $Z(p_*(E \otimes q^*(\cO_S(n))))$ is a linear polynomial, thus can be written as
$$Z(p_*(E \otimes q^*(\cO_S(n))))=a(E)n+b(E)+\sqrt{-1}(c(E)n+d(E))$$
for some group morphisms $a, b, c, d \colon \rK_0(X \times S) \to \bR$. Now by \cite[Lemma 4.1]{Liu}, for every positive $t \in \bR$,  the function
$$Z_t(E)=a(E)t -d(E)+\sqrt{-1}c(E)t$$
defines a \emph{weak stability function} on $\cA_S$ (meaning that the first condition in Definition~\ref{def:stabcond_Bridgeland} is relaxed to $\Im(Z_t(E))= 0$ implies $\Re(Z_t(E))\leq0$). Tilting $\cA_S$ with respect to $Z_t$, one gets a new heart $$\cA_S^t=\langle \cT^t, \cF^t[1] \rangle,$$ where $\cT^t$, resp.\ $\cF^t$, is generated by $Z_t$-semistable objects in $\cA_S$ with slope $>0$, resp.\ $\leq 0$. As shown in \cite[Proposition 4.6]{Liu}, for every $s, t \in \bR_{>0}$ the function
$$Z_S^{s,t}(E) =c(E)s+b(E)+\sqrt{-1}(-a(E)t+d(E))$$
defines a stability function on $\cA_S^t$. In fact, the following result holds.

\begin{Thm}[\cite{Liu}, Theorem 5.9] \label{thm:Liu}
Let $X$ be a smooth projective variety over an algebraically closed field $k$, and $S$ be a smooth projective curve over $k$. Let $\sigma$ be a stability condition on $\Db(X)$ whose heart is noetherian. 
For every $s, t \in \bR_{>0}$ the pair $\sigma_{s,t}=(\cA_S^t, Z_S^{s,t})$ defines a stability condition on $\Db(X \times S)$ with respect to $v'=(v_1, v_2) \colon \rK_0(X \times S) \to \Lambda'$, where $\Lambda'=\Lambda \oplus \Lambda/\ker(Z)$, and  
\begin{align*}
v_1(E)&=v(p_*(E \otimes q_*\cO_S(n))) -v(p_*(E \otimes q_*\cO_S(n-1))), \\   
v_2(E)&= v(p_*(E \otimes q_*\cO_S(n))) - nv_1(E).
\end{align*}
\end{Thm}
As an important consequence, this construction produces stability conditions on arbitrary products of curves $X=C_1 \times \dots \times C_m$ using slope stability on $C_i$. 

Note that $\sigma_{s,t}$ is a stability condition in the weaker sense of Definition~\ref{def:stabcond_Bridgeland}. A natural question is whether these stability conditions have moduli spaces with nice geometric properties.

\begin{Ques} \label{quest:stabcondproduct_proper}
Do the stability conditions in Theorem~\ref{thm:Liu} satisfy the conditions in Definition~\ref{definition-stability-families}?  
\end{Ques}
    
A recent application of Theorem~\ref{thm:Liu} is the construction of stability conditions on Hilbert schemes of points on product surfaces.

\begin{Thm} [\cite{LMPSZ}, Theorem 1.3] \label{thm:LMPSZ}
Let $X$ be a smooth projective surface over an algebraically closed field of characteristic $0$ which is either the product $C_1 \times C_2$ of smooth projective curves of positive genus, or the Kummer K3 surface associated to the abelian surface $E_1 \times E_2$ where $E_1$ and $E_2$ are elliptic curves. Then there exist stability conditions on $\Db(\Hilb^n(X))$. 
\end{Thm}
\begin{proof}[Idea of proof]
Consider the case $X=C_1 \times C_2$. By Theorem~\ref{thm:Liu} there are stability conditions on the $n$-fold product $X^n=(C_1 \times C_2)^n$. On the other hand, by \cite{BKR} there is an equivalence
$$\Db(\Hilb^n(X)) \simeq \Db([X^n/ \mathfrak{S}_n])= \Db(X^n)^{\mathfrak{S}_n}$$
where the right-hand-side is given by the stacky quotient of $X^n$ by the action of $\mathfrak{S}_n$ by permutations. By \cite{MMS:inducing} a stability condition on $\Db(X^n)$ induces a stability condition on the invariant category $\Db(X^n)^{\mathfrak{S}_n}$ if it is fixed by the action of $\mathfrak{S}_n$. The stability conditions from Theorem~\ref{thm:Liu} do not work directly, but expanding the construction in Theorem~\ref{thm:Liu}, the authors show in \cite[Theorem 4.5]{LMPSZ} the existence of $\mathfrak{S}_n$-fixed stability conditions on $X^n$ as desired. A similar argument applies to the Kummer surface of a product of elliptic curves.
\end{proof}

Theorem~\ref{thm:LMPSZ} is the first ingredient for proving the existence of stability conditions in a more general setting. Indeed, in the work in preparation \cite{MPS1} the authors prove foundational results  on deformations of t-structures, which combined with the above result lead to the following statement.

\begin{Thm}[\cite{LMPSZ2}] \label{thm:LMPSZ_2}
Let $X$ be a very general polarized hyperk\"ahler manifold of $\rK3$-type or Kummer type of dimension $2n$. Then there exist stability conditions on $\Db(X)$.  
\end{Thm} 

Answering Question~\ref{quest:stabcondproduct_proper} would be very useful to address the same problem for the stability conditions in Theorems~\ref{thm:LMPSZ}, \ref{thm:LMPSZ_2}.

\subsubsection{Update}

During the preparation of these notes, a groundbreaking progress in the construction of stability conditions has been obtained by Chunyi Li, who solved this problem in the case of smooth projective varieties in any dimension.

\begin{Thm}[\cite{Chunyi:remark}] \label{thm:chunyili}
Let $X$ be a smooth projective variety over $\bC$. Then there are stability conditions on $\Dperf(X)$.   
\end{Thm}
\begin{proof}[Idea of proof]
Recall that given a finite morphism $f \colon Y \to X$ between smooth projective varieties and stability conditions $\sigma_Y$ and $\sigma_X$ on $Y$ and $X$ respectively, satisfying certain conditions, there are induced stability conditions $f_\sharp\sigma_Y$ and $f^\sharp\sigma_X$ on $X$ and $Y$ \cite{Polishchuk:families-of-t-structures}.
Now let $E$ be an elliptic curve and consider the composition
$$\pi \colon E^n \to E^n/ (\bZ/2\bZ)^n \cong (\bP^1)^n \to (\bP^1)^n/\mathfrak{S}_n \cong \bP^n.$$
By Theorem~\ref{thm:Liu} there is a stability condition $\sigma$ on $E^n$. Combining the results in \cite{FLZ:finite_albanese, LMPSZ} it is possible to show that $\sigma$ is fixed by the action of $(\bZ/2\bZ)^n \rtimes \mathfrak{S}_n$, and to obtain a stability condition $\pi_\sharp\sigma$ on $\bP^n$. Next in \cite{Chunyi:realreduction} it was proved (very roughly speaking) that if one can define a family of stability conditions on $\bP^n$ satisfying certain properties, called Bayer property and Restriction-$N$ property, then any smooth projective subvariety of $\bP^n$ inherits stability conditions. These special requirements are checked for $\sigma$, and as a consequence are verified for $\pi_\sharp\sigma$. This implies that if $i \colon X \to \bP^n$ is a closed embedding, then $i^\sharp\pi_\sharp \sigma$ is a stability condition on $\Dperf(X)$.
\end{proof}

A natural question is whether these stability conditions have well-behaved moduli spaces.

\begin{Ques} \label{quest:chunyistabproper}
Do the stability conditions in Theorem~\ref{thm:chunyili} satisfy the conditions in Definition~\ref{definition-stability-families}?  
\end{Ques}

In the paper in preparation \cite{Ziqi2} the author shows that the stability conditions induced on $(\bP^1)^n$ satisfy the conditions in Definition~\ref{definition-stability-families}, making a first progress towards Question~\ref{quest:chunyistabproper}.

Another interesting direction would be to consider when $\Dperf(X)$ has a semiorthogonal decomposition, as for instance in the examples in Section~\ref{sec:examples}, and try to induce a stability condition on the semiorthogonal components. A first case would be that of cubic hypersurfaces. 
\begin{Ques} \label{quest:stabonku}
Let $Y \subset \bP^{n+1}$ be a cubic hypersurface of dimension $n$ over an algebraically closed field $k$. Consider the subcategory $\Ku(Y)=\langle \cO_Y, \cO_Y(1),  \dots, \cO_Y(n-2) \rangle^\perp$. Are there stability conditions on $\Db(Y)$ which restrict to stability conditions on $\Ku(Y)$?  
\end{Ques}
We discuss more details on stability conditions on the Kuznetsov component of cubic fourfolds and threefolds in the next section. Moreover, the case of cubic fivefolds has been recently treated in \cite{Peize}.
Answering Question~\ref{quest:stabonku} would provide the first construction of stability conditions on $\Ku(Y)$ where $Y$ is a cubic hypersurface of dimension $n> 5$. In the special case $n=7$, one would have stability conditions on a noncommutative CY threefold (see Example~\ref{ex:cubicseven}). Even the case of cubic fourfolds would be interesting. Indeed, the known stability conditions are constructed through an involved procedure which makes hard to study the geometry of the associated moduli spaces of semistable objects. A more natural approach using the stability conditions from Theorem~\ref{thm:chunyili} could simplify many arguments from \cite{LPZ:twistedcubics, LPZ:ellipticquintics}. 

\subsection{Hyperk\"ahler geometry} \label{sec:HKgeo} In this section, we work over $\bC$. A smooth projective variety $X$ is hyperk\"ahler if it is simply connected and the space of holomorphic two-forms on $X$ is one-dimensional, spanned by a symplectic form. The importance of this class of varieties is originally motivated by the Beuville--Bogomolov decomposition theorem, where HK manifolds appear as one of the building blocks of manifolds with trivial canonical bundle. Their study has deep connections to differential geometry and theoretical physics. 

One of the most intriguing problems is the classification of HK varieties. In dimension $2$, they are K3 surfaces, while producing examples in higher dimension is hard. Up to deformation equivalence, the following different examples are known by the works of Beauville \cite{Beauville:HK} and O'Grady \cite{OGrady, OG6}:
\begin{itemize}
\item for $n \geq 2$, the Hilbert scheme $S^{[n]}$ of points of length $n$ on a K3 surfaces $S$;
\item for $n \geq 2$, the generalised Kummer variety $\Kum^n(A)$ associated to an abelian surface $A$. It is defined as the kernel of the summation map composed with the Hilbert--Chow morphism $A^{[n+1]} \to A$.
\item a $10$-dimensional example OG10 constructed as a resolution of a singular moduli space of stable sheaves on a K3 surface;
\item a $6$-dimensional example OG6 constructed as a resolution of a singular generalized Kummer variety of an abelian surface.
\end{itemize}
A HK variety $X$ is of K3-type, Kummer-type, OG10-type or OG6-type if it is deformation equivalent to $S^{[n]}$, $\Kum^n(A)$, OG10 or OG6 as above. Let us explain the connection with noncommutative K3 or abelian surfaces as in Examples~\ref{ex:cubicfourfolds}, \ref{ex:GMvarieties}, \ref{ex:ncas}.

\subsubsection{K3-type HK manifolds}
 
By the combination of many works \cite{Mukai:BundlesK3, Beauville:HK, OG:weight2,  Yoshioka:Irreducibility, Yoshioka:Abelian} moduli spaces of stable sheaves on a K3 surface $S$ are HK manifolds of K3-type (see \cite{BM:projectivity} for the generalization to stable objects in $\Dperf(X)$ with respect to a stability condition). Since the second Betti number of $S^{[n]}$ is $23$, polarized HK varieties of dimension $2n$ of K3-type have $20$-dimensional moduli. However, those arising from moduli spaces of stable sheaves on K3 surfaces have $19$-dimensional moduli as polarized K3 surfaces. Besides some special case given by Fano varieties of lines on cubic fourfolds, Debarre--Voisin varieties and double EPW sextics, there were no other explicit descriptions of locally complete families of polarized HK manifolds of K3-type constructed through classical geometry. 

On the other hand, cubic fourfolds combine the properties of having $20$-dimensional moduli and a natural associated noncommutative K3 surface. In this situation, one can hope to produce locally complete families of polarized HK varieties of K3-type from moduli spaces of stable objects in the noncommutative K3 surface by constructing a stability condition as in Definition~\ref{definition-stability-families}. This is proved in \cite{BLMS:Kuznetsov, BLMNPS:families} and we summarize here the main results. 

Let $Y$ be a cubic fourfold. The numerical Grothendieck group of $\Ku(Y)$ is the quotient
$$\Knum(\Ku(Y))=\rK_0(\Ku(Y))/ \ker\chi$$
of the Grothendieck group of $\Ku(Y)$ by the kernel of the Euler pairing $\chi$ defined by
$$\chi(-,-)=\sum_i (-1)^i \dim\Hom(-,-[i]).$$
As computed in \cite{AddingtonThomas:CubicFourfolds}, there are two classes $\lambda_1$ and $\lambda_2 \in \Knum(\Ku(Y))$ spanning an $A_2$-lattice with respect to the pairing $(-,-)=-\chi(-,-)$:
$$
\Knum(\Ku(Y)) \supset \langle \lambda_1, \lambda_2 \rangle = \begin{pmatrix}
2 & -1 \\
-1 & 2
\end{pmatrix}.
$$
If $Y$ is very general, then $\Knum(\Ku(Y))=A_2$. 

Consider now a family of cubic fourfolds $Y \to S$ over a connected quasi-projective variety $S$. Similarly to Definition~\ref{def:grothgroup} or explicitly by \cite[Definition 21.1]{BLMNPS:families}, it is possible to define the relative numerical Grothendieck group $\Knum(\Ku(Y)/S)$ of $\Ku(Y)$. The constructed stability conditions on $\Ku(Y)$ over $S$ are \emph{numerical}, in the sense that they are defined with respect to a group morphism $\bv \colon \Knum(\Ku(Y)/S) \to A_2^\vee$, as below, and fiberwise are numerical stability conditions.

\begin{Thm}[\cite{BLMNPS:families}, Proposition 30.5] \label{thm:stabKufamily}
Let $Y \to S$ be a family of cubic fourfolds over a connected quasi-projective variety $S$. Then there is a stability condition $\sigma$ on $\Ku(Y)$ over $S$ with respect to $A_2^\vee$.    
\end{Thm}
\begin{proof}[Idea of proof]
By \cite[Lemma 30.4]{BLMNPS:families} there is a universally locally constant Mukai morphism $\bv \colon \Knum(\Ku(Y)/S) \to A_2^\vee$ such that for every closed point $s \in S$ the base change
$$\Knum(\Ku(Y_s)) \to \Knum(\Ku(Y)/S) \xrightarrow{\bv} A_2^\vee$$
is equal to the natural composition
$$\Knum(\Ku(Y_s)) \to \Knum(\Ku(Y_s))^\vee \to A_2^\vee.$$
If $S=s$ is a point, then the blow-up of a line in a cubic fourfold $Y_s$ which is not in a plane contained in $Y_s$, produces a conic fibration over a projective space $\bP^3_s$. In this case, the Kuznetsov component $\Ku(Y_s)$ embeds in the a twisted derived category $\Db(\bP^3_s, \cB_0^s)$. Stability conditions are constructed by restriction of tilt stability on $\Db(\bP^3_s, \cB_0^s)$ in \cite{BLMS:Kuznetsov}. This generalizes to the relative setting, giving by  \cite[Theorem 23.1]{BLMNPS:families} a stability condition $\sigma$ on $\Ku(Y)$ over $S$ with respect to a group morphism $\Knum(\Ku(Y)/S) \to \Lambda_1$, which is defined by the restriction of the twisted Chern character on the relative $\Db(\bP^3_S, \cB_0^S)$. \cite[Proposition 9.10]{BLMS:Kuznetsov} shows that $\Lambda_1=A_2^\vee$, and thus $\sigma$ is a stability condition on $\Ku(Y)$ over $S$ with respect to $\bv$.
\end{proof}

The following structure result on moduli spaces of stable objects can thus be obtained by deformation to the locus of cubic fourfolds with Kuznetsov component equivalent to the bounded derived category of a K3 surface.

\begin{Thm}[\cite{BLMNPS:families}, Theorems 29.1, 29.2]
Let $Y$ be a cubic fourfold.
\begin{itemize}
\item There is a non-empty connected component $\Stab^\dag(\Ku(Y))$ of stability conditions containing those constructed in Theorem~\ref{thm:stabKufamily}.
\item Let $v \in \Knum(\Ku(Y))$ be a primitive class and $\sigma \in \Stab^\dag(\Ku(Y))$ be a $v$-generic stability condition. Let $\phi \in \bR$ such that $Z(v) \in \bR_{>0} e^{i\pi\phi}$. Then the moduli space $\rM_{\sigma}(v,\phi)$ is non-empty if and only if  $v^2 \geq -2$. In this case, it is a smooth projective HK manifold of dimension $v^2+2$ of K3 -type.
\end{itemize}
\end{Thm} 

Choosing $v=a\lambda_1+b\lambda_2 \in A_2$ provides the following important application to the construction of locally complete families of polarized HK varieties.

\begin{Cor}[\cite{BLMNPS:families}, Corollary 29.5] \label{cor:HKfrommoduli_cubic4}
For any pair $a, b$ of coprime integers, there is a unirational locally complete $20$-dimensional family, over an open subset of the moduli space of cubic fourfolds, of polarized HK manifolds of K3-type of dimension $2n+2$, degree $6n$ and divisibility $2n$ if 3 does not divide $n$, or degree and divisibility $2 n/3$ otherwise, where $n=a^2-ab+b^2$. 
\end{Cor}

Given the success of the example of cubic fourfolds, it is natural to look for other situations where the above program can be carried out. This has been achieved in the case of Gushel--Mukai fourfolds and sixfolds and their Kuznetsov component as in Example~\ref{ex:GMvarieties}, leading to the following results.

\begin{Thm}[\cite{PPZ}, Theorems 1.4, 1.5]
Let X be a GM fourfold or sixfold.
\begin{itemize}
\item There is a non-empty connected component $\Stab^\dag(\Ku(X))$ of stability conditions on the Kuznetsov component $\Ku(X)$.
\item Let $v \in \Knum(\Ku(X))$ be a primitive class and $\sigma \in \Stab^\dag(\Ku(X))$ be a $v$-generic stability condition. Let $\phi \in \bR$ such that $Z(v) \in \bR_{>0} e^{i\pi\phi}$. Then the moduli space $\rM_{\sigma}(v,\phi)$ is non-empty if and only if  $v^2 \geq -2$. In this case, it is a smooth projective HK manifold of dimension $v^2+2$ of K3 -type.
\end{itemize}
\end{Thm}

If $X$ is a GM fourfold, then $\Knum(\Ku(X))$ contains an $A_1^{\oplus 2}$-lattice 
$$
\Knum(\Ku(X)) \supset \langle \lambda_1, \lambda_2 \rangle = \begin{pmatrix}
2 & 0 \\
0 & 2
\end{pmatrix}
$$
by \cite{Kuz-Perry:GM}. As an application, one obtains other infinite series of $20$-dimensional families of polarized HK manifolds of K3-type.

\begin{Cor}
For any pair $a, b$ of coprime integers, there is a unirational locally complete family, over an open subset of the moduli space of GM fourfolds, of polarized HK manifolds of K3 type of dimension $2(a^2 + b^2 + 1)$, degree $2(a^2 + b^2)$, and divisibility $a^2 + b^2$.    
\end{Cor}

\subsubsection{Modular interpretation of HK manifolds} \label{sec:modularHK}

Let $Y$ be a cubic fourfold. Beauville and Donagi showed that the Fano variety $F(Y)$ parametrizing lines in $Y$ is a projective HK fourfold of K3-type \cite{BeauDonagi}. Assuming that $Y$ does not contain a plane, in \cite{LLSvS:twistedcubics} Lehn, Lehn, Sorger and van Straten constructed a projective HK eightfold $Z(Y)$ of K3-type out of twisted cubic curves on $Y$.
The connection with the moduli spaces in Theorem~\ref{cor:HKfrommoduli_cubic4} is provided in the following result.

\begin{Thm}[\cite{LPZ:twistedcubics}]
Let $Y$ be a cubic fourfold. Then there is a stability condition $\sigma \in \Stab^\dag(\Ku(Y))$ such that $F(Y) \cong \rM_\sigma(\lambda_1+\lambda_2)$, and if $Y$ does not contain a plane, then $Z(Y) \cong \rM_\sigma(2\lambda_1+\lambda_2)$.   
\end{Thm}

It is natural to ask what happens considering rational curves of higher degree on $Y$. In \cite{LPZ:ellipticquintics} we address the case of rational quartic curves in $Y$, equivalently elliptic quintic curves in $Y$ by residuality. First, we study singular moduli spaces of semistable objects in $\Ku(Y)$ with numerical class $2v_0$ with $v_0^2=2$.
\begin{Thm}[\cite{LPZ:ellipticquintics}]
Let $Y$ be a cubic fourfold. Let $v_0 \in \Knum(\Ku(Y))$ primitive with $v_0^2=2$, and set $v=2v_0$.
If $\sigma \in \Stab^\dag(\Ku(Y))$ is $v$-generic, the moduli space $\rM_\sigma(v)$ has a symplectic resolution $\widetilde{\rM(v)}$ which is a $10$-dimensional smooth projective HK manifold of OG10-type.
\end{Thm}

Our main application is the construction of a compactification of the twisted fibration in intermediate Jacobians of $Y$. Indeed, consider the family $p \colon J \to U \subset (\bP^5)^\vee$ whose fibers are the twisted intermediate Jacobians of the smooth cubic threefolds parametrized by $U$. In \cite{LSV, Voisin:twistedcase, Giulia} Laza, Saccà and Voisin showed that there is a HK tenfold of OG10-type together with a Lagrangian fibration $\bar{J} \to (\bP^5)^\vee$ compactifying $p$.

\begin{Thm}[\cite{LPZ:ellipticquintics}]
Let $Y$ be a cubic fourfold and set $v_0=\lambda_1+\lambda_2$. Then there exists a projective HK manifold $N$ birational to $\widetilde{\rM(2v_0)}$ with a Lagrangian fibration compactifying $p \colon J \to U$.    
\end{Thm}

As an application, we obtain the following important consequence on the main component of the Hilbert scheme of rational quartic curves on $Y$, answering a conjecture of Castravet.

\begin{Cor}[\cite{LPZ:ellipticquintics}]
Let $M_4(Y)$ be the connected component of the Hilbert scheme which parametrizes rational quartic curves on $Y$. Set $v_0=\lambda_1+\lambda_2$. Then $\rM_\sigma(2v_0)$ is the maximally rationally connected quotient of $M_4(Y)$. In particular, the MRC quotient of $M_4(Y)$ is birational to the  twisted intermediate Jacobian fibration $J$.   
\end{Cor}

Next, consider the connected component $M_d(Y)$ of the Hilbert scheme of rational curves of degree $d$ on $Y$ with $d \geq 5$. If $Y$ is very general, then de Jong and Starr showed in \cite{dJStarr} that if $d$ is odd, then $M_d(Y)$ has a generically non-degenerate $2$-form, while if $d$ is even this form has generically a $1$-dimensional kernel. We formulate the following question.

\begin{Ques}
Let $Y$ be a very general cubic fourfold. If $d \geq 5$ is odd, is there a rational map $M_d(Y) \to \rM_\sigma(v)$ for some $v \in \Knum(\Ku(Y))$?   
\end{Ques}

Answering this question may be useful to compute the Kodaira dimension of $M_d(Y)$, or to get some interesting subvariety of $ \rM_\sigma(v)$.

Another interesting viewpoint is to study the derived category of moduli spaces $\rM_\sigma(v)$ of stable objects in $\Ku(Y)$. Let us first consider the case of $F(Y)$. Assume that $F(Y)$ is birational to $S^{[2]}$, where $S$ is a K3 surface. Then by \cite{BKR} there is an equivalence 
$$\Db(F(Y)) \simeq \Db(S^{[2]}) \simeq \Db(S)^{[2]},$$
where $\Db(S)^{[2]}=\Db(S^2)^{\mathfrak{S}_2}$ is the invariant category with respect to the action of the involution on the factors of the product $S \times S$. More generally, define 
$\Ku(Y) \boxtimes \Ku(Y)$ as the subcategory of $\Db(Y \times Y)$ of objects $E$ such that 
$$\Hom^\bullet(F \boxtimes \cO_Y(i), E)=\Hom^\bullet(\cO_Y(i) \boxtimes F, E)=0 \quad \text{for every } F \in \Db(Y), i=0,1,2.$$
The involution on $Y \times Y$ swapping the two factors induces a $\mathfrak{S}_2$-action on $\Ku(Y) \boxtimes \Ku(Y)$, and we denote by $\Ku(Y)^{[2]}$ the associated invariant category. The result below proves a long standing conjecture by Galkin.

\begin{Thm}[Galkin's conjecture, \cite{KimoiSegal}] \label{thm:Galkin}
Let $Y$ be a cubic fourfold. Then there is an equivalence $$\Db(F(Y)) \simeq \Ku(Y)^{[2]}.$$   
\end{Thm}

In \cite{BottHuy} Bottini and Huybrechts provided an alternative proof of Theorem~\ref{thm:Galkin} assuming $F(Y)$ has a rational Lagrangian fibration and is generic with this property. In this case, there is a twisted K3 surface realizing an equivalence $\Ku(Y) \simeq \Db(S, \alpha)$ and $F(Y)$ is birational to a moduli space of $\alpha$-twisted sheaves on $S$. The argument relies on proving a twisted version of BKR-equivalence, as well as studying the derived category of moduli spaces of twisted stable sheaves of rank $1$ compactifying relative twisted Jacobian fibrations of curves in a linear system on $S$. Motivated by the above success, we formulate the following naive problem (see \cite[Speculation 1.1]{Zhang}).

\begin{Ques} \label{quest:derivedcatmoduli}
Let $Y$ be a cubic fourfold. Let $\rM_\sigma(v)$ be a moduli space of stable objects in $\Ku(Y)$ of dimension $2n$ with a rational Lagrangian fibration. Is there a version of Galkin's conjecture for $\rM_\sigma(v)$ and $\Ku(Y)^{[n]}$?
\end{Ques}

In the paper in preparation \cite{MoritzSaket} the authors answer to Question~\ref{quest:derivedcatmoduli} assuming $n-1$ is square free and $\rM_\sigma(v)$ has Picard rank two. In the special case of the LLSvS eightfold, they obtain the following result.

\begin{Thm}[\cite{MoritzSaket}, Corollary 8.3]
Let $Y$ be a cubic fourfold with $\rk\Knum(\Ku(Y))=3$. Assume that $Z(Y)$ admits a
rational Lagrangian fibration. Then there is a Brauer class $\theta \in \Br(Z(Y))$ and an equivalence
$$\Db(Z(Y), 4\theta) \simeq \Ku(Y)^{[4]}.$$
\end{Thm}

\subsubsection{Kummer-type HK manifolds}

The construction of the HK variety $\Kum^n(A)$ associated to an abelian surface has the following generalization due to Yoshioka \cite{Yoshioka:Abelian} using moduli spaces of stable sheaves on $A$. Denote by $\rH^\ev(A, \bZ)$ the integral even cohomology of $A$. Endowed with the Mukai pairing it is isomophic to the lattice $U^{\oplus 4}$, and carries a natural weight $2$ Hodge structure. Let $v$ be a primitive Hodge class in $\rH^\ev(A, \bZ)$ such that $v^2 \geq 6$. Let $\rM_H(v)$ be the moduli space of Gieseker stable sheaves with respect to a fixed polarization $H$ having Chern character $v$. As in the case of K3 surfaces, by Mukai's result, $\rM_H(v)$ is a smooth projective variety of dimension $v^2+2$. Consider the morphism
$$a_v \colon \rM_H(v) \to A \times \hat{A}$$
defined by a translation of $\hat{\det} \times \det$, where $\det$ associates to a sheaf its determinant. By \cite{Yoshioka:Abelian} we have that $a_v$ is the Albanese morphism and the kernel $K_v(A):=a_v^{-1}(0)$ is a HK variety of Kummer-type. A similar result holds for moduli spaces of stable objects in $\Db(A)$. 

Nevertheless, generalized Kummer varieties associated to abelian surfaces have $3$-dimensional moduli. Since their second Betti number is $7$, polarized HK manifolds of Kummer-type have $4$-dimensional moduli. So far, there are no explicit examples of locally complete families of polarized HK manifolds of Kummer type. 

In \cite{BPPZ} we address this problem by constructing Kummer-type polarized HK manifolds out of moduli spaces of stable objects in the noncommutative abelian surfaces introduced in Example~\ref{ex:ncas}. We summarize our results as follows.

\begin{Thm}[\cite{BPPZ}] \label{thm:BPPZ}
Let $A$ be an abelian surface, and let $L \subset \rH^\ev(A, \bZ)$ be a primitive positive definite rank two lattice of Hodge classes.
\begin{enumerate}
\item The construction in Example~\ref{ex:ncas} (for a choice of $\Lambda$ determined by $L$) defines a noncommutative CY$2$ variety $\cA$ over a four-dimensional base scheme $B$ with a quasi-finite map to the global period domain of $L$ and a base point $b_0$ such that $\cA_{b_0} \simeq \Db(A)$.
\item There is a stability condition $\sigma$ on $\cA$ over $B$ such that its central charge is given by pairing with classes in $L$.
\item For every $v \in L$ and every $b \in B$ the moduli space $\rM_{\sigma_b}(v)$ of $\sigma_b$-stable objects in $\cA_b$ is non-empty.
\item For every $b \in B$ the fiber of the Albanese morphism of $\rM_{\sigma_b}(v)$ is a HK manifold of dimension $v^2-2$ of Kummer-type.
\end{enumerate}
\end{Thm}

\begin{Cor} \label{cor:kummerHK}
Let $(K_v(A), h)$ be a generalized Kummer variety with polarization $h$. Then there is an open subset of the moduli space of polarized HK manifolds of Kummer-type containing $(K_v(A), h)$, whose closed points correspond to fibers of Albanese morphisms of some moduli space $\rM_{\sigma_b}(v)$ of stable objects in $\cA_b$ as in Theorem~\ref{thm:BPPZ}.    
\end{Cor}

The construction in Thereom~\ref{thm:BPPZ} could be useful to recover OG6-type polarized HK manifolds. In this case, the OG6 examples have codimension-two in the moduli space. Answering the following question would describe codimension-one loci in the moduli space of polarized OG6-type HK manifolds.

\begin{Ques}
Using the notation from Theorem~\ref{thm:BPPZ}, let $v_0 \in L$ with $v_0^2=2$. Is there a symplectic resolution of the singular moduli space $\rM_{\sigma_b}(2v_0)$ such that the resolution of the fiber of the Albanese   morphism is a HK variety of OG6-type?
\end{Ques}

More generally, it would be interesting to produce new examples of noncommutative CY$2$ varieties with stability conditions: their moduli spaces could be source of families of HK manifolds. Progress in this direction are obtained in \cite{MPS2} by constructing a deformation of the derived category of a K3 surface in the same spirit as \cite{MaMeh}.

On the other hand, noncommutative CY$1$ varieties with stability conditions are equivalent to the derived category of an elliptic curve by \cite{Sung}. In the opposite direction, seeking for a classification result in the CY$2$ case, we formulate the following question.

\begin{Ques}
Let $\cC$ be a noncommutative CY$2$ variety with a stability condition. Is $\cC$ deformation equivalent to the derived category of a K3 or abelian surface?    
\end{Ques}

\subsection{Nocommutative curves} \label{sec:nccurves}

In Example~\ref{ex:primeFano3} we have introduced the noncommutative smooth and proper variety $\Ku(X)$ associated to a Fano threefold $X$ with Picard rank one and index one or two. By \cite{BLMS:Kuznetsov, BLMNPS:families} numerical stability conditions exist on $\Ku(X)$, making possible to study moduli spaces of semistable objects. In fact, small dimensional moduli spaces have been directly studied and related to Hilbert schemes of rational curves such as lines in $X$, leading to interesting applications in classical geometry and categorical Torelli problems. See \cite{PS:survey} for a survey on this topic.

A natural question is then to understand the geometry of moduli spaces of semistable objects, e.g.\ when they are non-empty. If $X$ has index $2$ and degree $4$ or index $1$, degree $12$, $16$, $18$, then $\Ku(X) \simeq \Db(C)$ where $C$ is a curve of genus $\geq 2$. In this case, moduli spaces of $\sigma$-semistable objects are moduli spaces of slope semistable sheaves on $C$ and  the non-emptiness is well-known. 

In the other cases, the main difficulty is that $\Ku(X)$ is not equivalent to the derived category of a smooth projective variety, thus results on moduli spaces cannot be deduced by deformation to the geometric locus.

In \cite{PPZ_enriques} we develop a strategy to show the non-emptiness of moduli spaces of semistable objects in Enriques categories, which applies to the Kuznetsov component of a quartic double solid (index $2$ and degree $2$) and of a GM threefold (index $1$ and degree $10$). For interesting results on moduli spaces of stable objects in Enriques categories and relation with geometric constructions we suggest \cite{Ziqi}.

Next we show the following non-emptiness result, which provides a unified argument for all cases. Below $\sigma$ is a stability condition  constructed in \cite{BLMS:Kuznetsov} (or more generally a Serre-invariant stability condition, see \cite{PY, PR, FP}).

\begin{Thm}[\cite{LPZ:higher}] \label{thm:Fanononemptymoduli}
Let X be a prime Fano threefold with index $1$, degree $10 \leq d \leq  18$, or index $2$, degree  $d \leq 4$. For every non-zero character $v \in \Knum(\Ku(X))$, the moduli space $\rM_\sigma(v)$ is non-empty.    
\end{Thm}
\begin{proof}[Idea of proof]
By \cite{Kuz:3folds} we have that $\Knum(\Ku(X)) \cong \bZ^2$. The key point is then the following observation. Let $v, w \in \Knum(\Ku(X))$ and let $E_v$ and $E_w$ be two $\sigma$-stable objects with phases satisfying $\phi(E_w)- \phi(E_v) \in (0,1)$. Assume there exists a non-trivial extension $f \in \Hom(E_w, E_v[1])$ and let $E_f$ be the extended object given by the triangle
$$E_v \to E_f \to E_w \to E_v[1].$$
By definition $E_f$ has numerical class $v+w$. Note that the destabilizing factors of $E_f$ must lie in the parallelogram  spanned by $v$ and $w$. But if the latter does not contain points corresponding to classes in $\Knum(\Ku(X)) \cong \bZ^2$, equivalently when its area $v \times w=1$, then $E_f$ has to be stable.

Now for every primitive $v \in \Knum(\Ku(X))$ with norm in $\bZ^2$ greater than $1$, there exists a unique pair of classes $v_+$ and $v_-$ such that $v=v_++v_-$ and $v_- \times v_+=1$. Using the numerical properties of $\chi$ on $\Knum(\Ku(X))$ it is possible to show that $\chi(v_+, v_-)<0$ if the dimension of $\rM_\sigma(v)$ is smaller than a certain number depending on $\chi$. This condition together with the properties of $\sigma$ ensures the existence of a non-trivial extension $E_f$ between stable objects of class $v_-$ and $v_+$, which is stable by the above observation. This reduces the statement to check the non-emptiness of $\rM_\sigma(v)$ for a finite number of classes $v$ corresponding to moduli spaces of small dimension.
\end{proof}

In the case of cubic threefolds we further show the following results on the geometry of moduli spaces.

\begin{Thm}[\cite{LPZ:higher}]
Let $Y$ be a cubic threefold. For every $v \in \Knum(\Ku(Y))$ the moduli space $\rM_\sigma(v)$ is irreducible. Assume in addition that $v$ is primitive. Then $\rM_\sigma(v)$ is smooth projective of the expected dimension $1-\chi(v,v)$.
\end{Thm}

We also obtain the following result on the geometry of fibers of Abel--Jacobi maps associated to moduli spaces. More precisely, denote by
\[J(Y):= \frac{H^{2,1}(Y)^*}{H_3(Y,\mathbb Z)}\]
the intermediate Jacobian of $Y$, which is an abelian fivefold (playing a crucial role in Clemens and Griffiths proof of irrationality of cubic threefolds). Choosing a base point $F_0$, we define the Abel--Jacobi map
$$\Phi_v \colon \rM_\sigma(v) \to J(Y), \quad F\mapsto c_2(F)-c_2(F_0)$$
using the (cycle-theoretic) second Chern class.

\begin{Thm}[\cite{LPZ:higher}] \label{thm:cubic3folds}
Let $Y$ be a cubic threefold. Let $v$ be a primitive class in $\Knum(\Ku(Y))$.
\begin{enumerate}
\item If $\dim \rM_\sigma(v)> 5$, then the Abel--Jacobi map $\Phi_v \colon \rM_\sigma(v)\to J(Y)$ is surjective with connected fibers. For a general point $c\in J(Y)$, the fiber $\rM_\sigma(v,c)$ is a smooth Fano variety with primitive canonical divisor class.
\item Assume that $w \in \Knum(\Ku(Y))$ is primitive and that $\dim \rM_\sigma(v)\geq 23$, $\dim \rM_\sigma(w)\geq 23$. Then for general points $c,c'\in J(Y)$, the fibers $\rM_\sigma(v,c)$ and $\rM_\sigma(w,c')$ are stably birational equivalent.    
\end{enumerate}       
\end{Thm}

This shows that moduli spaces of stable objects in $\Ku(Y)$ behave similarly to moduli spaces of vector bundles on a curve of genus $g\geq 2$. This motivates the wording ``noncommutative curve'' when referring to $\Ku(Y)$.

As a consequence of Theorem~\ref{thm:cubic3folds} we obtain the following statement.

\begin{Cor}
Let $Y$ be cubic threefold. Then for every primitive $v\in\Knum(\Ku(Y))$ with $\chi(v,v)\leq -4$, the maximal rationally connected quotient of $\rM_\sigma(v)$ is the intermediate Jacobian $J(Y)$.
\end{Cor}

The above result may have interesting applications in the study of Hilbert schemes of curves on cubic threefolds. Indeed, if one can show the stability of a certain twist of the ideal sheaf of a smooth irreducible curve $C$ on $Y$ or of its projection in $\Ku(Y)$, then one would deduce a birational description of the component of the Hilbert scheme containing $C$ as a moduli space of stable objects, and that its maximally rationally connected quotient is the intermediate Jacobian.

For technical reasons, we had to impose the condition that the moduli spaces have dimension $\geq 23$ in the second part of Theorem~\ref{thm:cubic3folds}. Since the general fiber of eight-dimensional moduli spaces is birational to the cubic threefold $Y$, it would be very interesting to improve this result if possible, answering to the following question.

\begin{Ques}
Let $Y$ be a cubic threefold. Let $v \in \Knum(\Ku(X))$ be a primitive class. Is the general fiber of the Abel--Jacobi map $\Phi_v$ stably birational to $Y$?    
\end{Ques}

Other interesting cases to study from this perspective are quartic double solids and GM threefolds. We expect that the strategy can be adapted to this setting, although there are numerical differences which make us to expect that the fiber of the Abel--Jacobi map should have numerically trivial canonical class.

\begin{Ques}[\cite{LPZ:higher}, Question 8.11]
Let $X$ be a quartic double solid or a Gushel--Mukai threefold. Let $v \in \Knum(\Ku(X))$ be a primitive class. What is the Kodaira dimension of a general fiber of the Abel--Jacobi map $\Phi_v$?
\end{Ques}

Going back to the non-emptiness of moduli spaces, it would be interesting to generalize the strategy in Theorem~\ref{thm:Fanononemptymoduli} to noncommutative curves with numerical Grothendieck group of rank $3$. This would be useful to answer the following question.

\begin{Ques}
Let $X$ be a prime Fano threefold of index $1$ and degree $\in \lbrace 2, 4, 6, 8 \rbrace$. For which $v \in \Knum(\Ku(X))$ is the moduli space $\rM_\sigma(v)$ non-empty? If $v$ is primitive, is the moduli space $\rM_\sigma(v)$ non-empty if and only if the expected dimension $1-\chi(v,v) \geq 0$?
\end{Ques}

\bibliographystyle{halpha}    
\bibliography{all}                      

\end{document}